\documentclass[11pt]{article}
\usepackage{geometry}
\usepackage[parfill]{parskip}
\usepackage{graphicx}

\usepackage{amssymb}

\usepackage{epstopdf}

\usepackage{amsmath,amsthm,amscd,amssymb}
\usepackage{latexsym}
\numberwithin{equation}{section}

\theoremstyle{plain}
\newtheorem{theorem}{Theorem}[section]
\newtheorem{lemma}[theorem]{Lemma}
\newtheorem{corollary}[theorem]{Corollary}
\newtheorem{proposition}[theorem]{Proposition}

\theoremstyle{definition}

\theoremstyle{remark}

\newtheorem{case[theorem]}{Case}

\title{On the Erd{\H o}s distinct distance problem in the plane}
\author{Larry Guth  and Nets Hawk Katz}

\begin{document}

\maketitle

\begin{abstract} In this paper, we prove that a set of $N$ points in ${\bf R}^2$ has at least $c{N \over \log N}$ distinct
distances, thus obtaining the sharp exponent in a problem of Erd{\H o}s. We follow the set-up of Elekes and Sharir which, in
the spirit of the Erlangen program, allows us to study the problem in the group of rigid motions of the plane. This
converts the problem to one of point-line incidences in space. We introduce two new ideas in our proof.  In order to
control points where many lines are incident, we create a cell decompostion using the polynomial ham sandwich theorem. This
creates a dichotomy: either most of the points are in the interiors of the cells, in which case we immediately get
sharp results, or alternatively the points lie on the walls of the cells, in which case they are in the zero set of
a polynomial of suprisingly low degree, and we may apply the algebraic method. In order to control points where only two
lines are incident, we use the flecnode polynomial of the Rev. George Salmon to conclude that most of the lines lie on a ruled
surface. Then we use the geometry of ruled surfaces to complete the proof.
\end{abstract}

\section{Introduction}

In \cite{E}, Paul Erd{\H o}s posed the question: how few distinct distances are determined by $N$ points in the plane.
Erd{\H o}s checked that if the points are arranged in a square grid,
then the number of distinct distances is $\sim {N \over \sqrt{\log N}}$.  He conjectured that for any arrangement of $N$ points, the number of
distinct distances is $\gtrsim {N \over \sqrt{\log N}}$.  (Throughout this
paper, we use the notation $A \gtrsim B$ to mean that there is a universal constant $C > 0$ with $A > C B$.)

In the present paper, we prove

\begin{theorem} \label{main}  A set of $N$ points in the plane determines $\gtrsim {N \over \log N}$ distinct
distances. \end{theorem}

Various authors have proved lower bounds for the number of distinct distances. These include but are not limited to
\cite{M}, \cite{CSzT}, \cite{Sz}, \cite{SoTo},\cite{T}.  The most recent lower bound, in \cite{KT}, says that the number of distances
is $\gtrsim N^{.8641}$.  For a more thorough presentation of the history of the 
subject see the forthcoming book \cite{GIS}.

In \cite{ES}, Elekes and Sharir introduced a completely new approach to the distinct distance problem, which uses the symmetries of the problem in a novel way.  They laid out
a plan to prove Theorem \ref{main}, which we follow in this paper.  Their approach connects the distinct distance problem to
three-dimensional incidence geometry.  Using their arguments, Theorem 1.1 follows from the following estimate about the incidences
of lines in $\mathbb{R}^3$.

\begin{theorem} \label{introincid} Let $\frak L$ be a set of $N^2$ lines in $\mathbb{R}^3$.  Suppose that $\frak L$ contains
$\lesssim N$ lines in any plane or any regulus.  Suppose that $2 \le k \le N$.  Then the number of points that
lie in at least $k$ lines is $\lesssim {N^3 k^{-2}}$.
\end{theorem}

Recently, there has been a lot of progress in incidence geometry coming from the polynomial method.  In \cite{D}, Dvir 
used the polynomial method to prove the finite field Kakeya conjecture, which can be considered as a problem in
incidence geometry over finite fields.  In \cite{GK}, the polynomial method was applied to incidence geometry problems in
$\mathbb{R}^3$, solving the joints problem.  The method was simplified and generalized in
\cite{EKS}, \cite{KSS}, and \cite{Q}.  Kaplan, Sharir, and Shustin (\cite{KSS}) and Quilodr\'an (\cite{Q}) solved the joints problem in higher dimensions.  
For context, we mention here the joints theorem in $n$ dimensions.

\begin{theorem} \label{jointsndim} (\cite{KSS}, \cite{Q}) Let $n \ge 3$.  Let $\frak L$ be a set of $L$ lines in $\mathbb{R}^n$.  A joint of $\frak L$ is a point that
lies in $n$ lines of $\frak L$ with linearly independent directions.  The number of joints of $\frak L$ is $\le C_n L^{\frac{n+1}{n}}$.
\end{theorem}

One of the remarkable things about the polynomial method is how short the proofs are.  The finite field Kakeya problem and the joints
problem were considered to be very difficult, and many ideas were tried in both cases.  The proof of the finite field Kakeya result (\cite{D})
and the simplified proof of the joints theorem (\cite{KSS} or \cite{Q}) are each about one page long.  This simplicity
gives the feeling that these are the ``right" proofs for these theorems.

In \cite{EKS}, Elekes, Kaplan, and Sharir used the polynomial method to prove the case $k=3$ of Theorem \ref{introincid}.  (It is a special case of Theorem 9 of that paper.)
It remains to prove Theorem \ref{introincid} when $k=2$ and when $k$ is large.  This requires two new ideas.  When $k=2$, the key new
idea is an application of ruled surfaces.  When $k$ is large, the key new idea is an application of a ham sandwich theorem from topology.
Let's discuss these new ideas in more detail.

First we explain the extra difficulty that occurs when $k=2$, as opposed to $k=3$.  The fundamental idea of the polynomial method is to find
a polynomial $p$ of controlled degree whose zero set $Z$ contains the set of lines $\frak L$.  
Then one uses the geometry of $Z$ to study $\frak L$.  A point where three lines of 
$\frak L$ intersect is an unusual point of the surface $Z$: it is either a critical point of $Z$ or else a `flat' point of $Z$.  
One can use algebraic geometry to control the critical points and flat points of $Z$ in terms of the degree of $p$.  But a point of $Z$ where 
two lines of $\frak L$ intersect does not have to be either critical or flat, and we don't know of any special property of such a point.

Reguli play an important role in the case $k=2$.  If $l_1$, $l_2$, and $l_3$ are three pairwise-skew lines in ${\bf R}^3$,
there is a 1-parameter family of lines that intersects all three lines.  The union of the lines in the 1-parameter family is
a surface called a regulus.  A regulus is a degree 2 algebraic surface.  An example is the surface defined by the equation
$z = xy$.  A set of $N^2$ lines in a regulus can have $\sim N^4$ points of intersection.  A regulus is an example of a ruled
surface.  In this paper, a ruled surface means an algebraic surface which contains a line through each point.

We apply the theory of ruled surfaces to prove our estimate in the case $k=2$.  We first observe that if a surface $Z$ of controlled degree contains too many lines, then some component of the surface $Z$ must be ruled.  In this way, we can reduce to the
case of a set of lines contained in a ruled surface of controlled degree.  
Ruled surfaces have some special structure, and we use that structure to bound the intersections between
the lines.  A ruled surface is called singly-ruled if a generic point in the surface lies in only one line in the surface.  It is well known that
planes and reguli are not singly-ruled, but every other irreducible ruled surface is.  The reason
is that if a surface is not singly-ruled, it is easy to find three lines $l_1,l_2,$ and $l_3$
in the surface which meet infinitely many lines not at one of the possibly three points of intersection of $l_1,l_2,$ and $l_3$. This implies that the surface has a factor which is a
plane, if any two of $l_1,l_2,$ and $l_3$ is coplanar and a regulus if they are pairwise
skew. A point where two lines intersect inside of a singly-ruled surface must be critical - except for points lying in the union of a controlled number
of lines and a finite set of additional exceptions.  By using this type of result, the structure of ruled surfaces helps us to prove our estimate.

Next we try to explain the extra difficulty that occurs for large $k$.  An indication of the difficulty is that for large $k$, Theorem \ref{introincid} does
not hold over finite fields.  (When $k=2$ or $3$, it's an open question whether Theorem \ref{introincid} holds over finite fields, but we suspect that it does.)  
The counterexample occurs when one considers $\frak L$ to be all of the lines in $\mathbb{F}^3$.  (Here, $\mathbb{F}$ denotes a finite field.)
This situation is reminiscent of the situation for the Szemer\'edi-Trotter theorem.

The Szemer\'edi-Trotter incidence theorem (\cite{SzT}) is the most fundamental and important result in extremal incidence geometry.
It was partly inspired by Erd{\H o}s's distance problem, and it has played a role in all the recent work on the subject. 

\begin{theorem} \label{Szt} (Szemer\'edi-Trotter) Let $\frak L$ be a set of $L$ lines in $\mathbb{R}^2$.  Then the number of points
that lie in at least $k$ lines is $\le C (L^2 k^{-3} + L k^{-1})$.

\end{theorem}

The Szemer\'edi-Trotter theorem is also false over finite fields: the counterexample occurs when one considers
$\frak L$ to be all of the lines in $\mathbb{F}^2$.  All of the proofs of the theorem involve in some way the topology of
$\mathbb{R}^2$.  One approach, which is important in our paper, is the cellular method introduced in the seminal
paper  \cite{CEGSW} by Clarkson, Edelsbrunner, Guibas, Sharir, and Welzl.  The cellular method is a kind of divide-and-conquer
argument.  One carefully picks some lines, which divide the plane into cells, and then one studies $\frak L$
inside of each cell.

The cellular method has been very successful for problems in the plane, but only partly successful in higher dimensions.
For example, in \cite{FS}, Feldman and Sharir attacked the (3-dimensional) joints problem using the cellular method
(among other tools).  They were able to prove that the number of joints determined by $L$ lines is $\lesssim L^{1.62}$.
(For contrast, the algebraic method gives $\lesssim L^{3/2}$.)

It seems to us that there are strong analogies between Theorem \ref{introincid} and the Szemer\'edi-Trotter theorem, and
also between Theorem \ref{introincid} and the joints theorem.  As in the Szemer\'edi-Trotter theorem,  topology must play some role. As in the joints theorem, it is natural for polynomials 
to play some role.

To prove Theorem \ref{introincid} when $k$ is large, we construct a cell decomposition where the walls of the cells form an algebraic surface $Z$ defined
by a polynomial $p$.
The polynomial is found by a topological argument, using the polynomial version of the ham sandwich theorem.
At this point, our argument involves a dichotomy.  Let $\frak S$ denote the points that lie in at least $k$ lines of $\frak L$.  
In one extreme case, the points of $\frak S$ are evenly distributed
among the open cells of our decomposition.  In this case, we prove our estimate by the cellular method, similar to arguments from 
 \cite{CEGSW}.  In another extreme case, the points of $\frak S$ all lie in $Z$.   
 In this case, it turns out that the lines of $\frak L$ also
 lie in $Z$.  In this case, we prove our estimate by the polynomial method, studying the critical and flat points of  $Z$
 as in \cite{GK} or \cite{EKS}.
 
In Section 2, we explain the plan laid out by Elekes and Sharir.  In particular, we explain how Theorem \ref{main} follows
from Theorem \ref{introincid}.  In Section 3, we prove Theorem \ref{introincid} in the case $k=2$ using the ruled surfaces
method.  We begin with the necessary background on ruled surfaces.  In Section 4, we prove Theorem \ref{introincid} for
$k \ge 3$ using the polynomial cell method.  We begin with background on the polynomial ham sandwich theorem.
In an appendix, we follow how our argument plays out when the set of points is a square grid.  This example shows that several of our estimates
are sharp up to constant factors, including Theorem \ref{introincid}.

{\bf Acknowledgements:} The first author is partially supported by NSERC, by NSF grant DMS-0635607, and by the Monell
Foundation.  The second author is partially supported by NSF grant DMS-1001607. He would like to
thank Michael Larsen for some very helpful discussions about algebraic geometry. He would also like to thank the
Institute of Advanced Study for the use of its magnificent duck pond during a visit which resulted in this paper. Both authors would like to thank the helpful referee because of whom the
exposition in the paper is significantly improved.

\section{Elekes-Sharir framework}

Elekes and Sharir \cite{ES} developed a completely new approach to the distinct distance problem, connecting it to 
incidence geometry in 3-dimensional space.  In this section, we present (a small variation of) their work.

Let $P \subset {\bf R^2}$ be a set of $N$ points.  We let $d(P)$ denote the set of non-zero distances among points of
$P$.

$$ d(P) := \{ d(p,q) \}_{p,q \in P, p \not= q}. $$

To obtain a lower bound on the size of $d(P)$, we will prove
an upper bound on a set of quadruples.  We let $Q(P)$ be the set of quadruples,  $(p_1,p_2,p_3,p_4) \in P^4$
satisfying

\begin{equation}  \label{quadruple} d(p_1,p_2)=d(p_3,p_4) \not= 0. \end{equation}

We refer to the elements of $Q(P)$ as distance quadruples.  If $d(P)$ is small, then $Q(P)$ needs to be large.
By applying the Cauchy-Schwarz inequality, we easily obtain the following inequality.

\begin{lemma}  \label{CS} For any set $P \subset {\bf R^2}$ with $N$ points, the following inequality holds.

$$ |d(P)| \geq {N^4 - 2N^3 \over |Q(P)|}.$$
 
\end{lemma}

\begin{proof} Consider the distances in $d(P)$, which we denote by $d_1, ..., d_m$ with $m = | d(P) |$.  
There are $N^2 - N$ ordered pairs $(p_i, p_j) \in P^2$ with $p_i \not= p_j$.  Let $n_i$ be the number of 
these pairs at distance $d_i$.  So $\sum_{i=1}^m n_i = N^2 - N$.

The cardinality $|Q(P)|$ is equal to $\sum_{i=1}^m n_i^2$.  But by Cauchy-Schwarz,

$$ (N^2 - N)^2 = \left( \sum_{i=1}^m n_i \right)^2 \le \left( \sum_{i=1}^m n_i^2 \right) m 
= |Q(P)| |d(P)| . $$

Rearranging, we see that $|d(P)| \ge (N^2 - N)^2 |Q(P)|^{-1}$.  \end{proof}

To prove Theorem \ref{main}, it suffices to prove the following upper bound on $|Q(P)|$.  

\begin{proposition} \label{quadbound} For any set $P \subset {\bf R^2}$ of $N$ points, the number of
quadruples in $Q(P)$ is bounded by $|Q(P)| \lesssim N^3 \log N$.
\end{proposition}

This Proposition is sharp up to constant factors when $P$ is a square grid (see the appendix).

Elekes and Sharir study $Q(P)$ from a novel point of view related to the symmetries of the plane.  We let $G$ denote
the group of {\it positively oriented} rigid motions of the plane.  The first connection between $Q(P)$ and $G$ comes
from the following simple proposition.

\begin{proposition} \label{rigidmotion} Let $(p_1, p_2, p_3, p_4)$ be a distance quadruple in $Q(P)$.  Then there
is a unique $g \in G$ so that $g(p_1) = p_3$ and $g(p_2) = p_4$.
\end{proposition}

\begin{proof}  All positively oriented rigid motions taking $p_1$ to $p_3$
can be obtained from the translation from $p_1$ to $p_3$ by applying a rotation $R$ about the point $p_3$.  Since $d(p_3, p_4)
= d(p_1, p_2) >  0$, there is a
unique such rotation sending  $p_2 + p_3 - p_1$ into $p_4$. \end{proof}

Using Proposition \ref{rigidmotion}, we get a map $E$ from $Q(P)$ to $G$, which associates to each distance quadruple $(p_1,
p_2, p_3, p_4) \in Q(P)$, the unique $g \in G$ with $g(p_1) = p_3$ and $g(p_2) = p_4$.  The letter $E$ here stands for
Elekes, who introduced this idea.

Our goal is to use the map $E$ to help us estimate $|Q(P)|$ by counting appropriate rigid motions.  It's important to note that
the map $E$ is not necessarily injective.  The number of quadruples in $E^{-1}(g)$ depends on the size of $P \cap gP$.
We make this precise in the following lemma.

\begin{lemma} \label{inverseimage} Suppose that $g \in G$ is a rigid motion and that $| P \cap g P | = k$.  Then the number of
quadruples in $E^{-1}(g)$ is $2 {k \choose 2}$.
\end{lemma}

\begin{proof} Suppose that $P \cap gP$ is $\{ q_1, ..., q_k \}$.  Let $p_i = g^{-1}(q_i)$.  Since $q_i$ lies in $gP$, each point $p_i$
lies in $P$.  For any ordered pair $(q_i, q_j)$ with $q_i \not= q_j$, the set $(p_i, p_j, q_i, q_j)$ is a distance quadruple.  This assertion is easy to check.
We have seen that $p_i, p_j, q_i, q_j$ all lie in $P$.  Since $g$ preserves distances, $d(p_i, p_j) = d(q_i, q_j)$.  Since $q_i \not= q_j$,
the distance $d(q_i, q_j) \not= 0$.  

Now we check that every distance quadruple in $E^{-1}(g)$ is of this form.  Let $(p_1, p_2, p_3, p_4)$ be a distance quadruple
in $E^{-1}(g)$.  We know that $g(p_1) = p_3$ and $g(p_2) = p_4$.  So $p_3, p_4$ lie in $P \cap gP$.  Say $p_3 = q_i$ and
$p_4 = q_j$.  Now $p_1 = g^{-1} (p_3) = p_i$ and $p_2 = g^{-1}(p_4) = p_j$.  \end{proof}

Let $G_{=k}(P) \subset G$ be the set of $g \in G$ with $|P \cap gP| = k$.  Notice that $G_{=N}(P)$ is a subgroup of $G$.  It is the group of orientation-preserving
symmetries of the set $P$.  For other $k$, $G_{=k}(P)$ is not a group, but these sets can still be regarded as sets of
``partial symmetries" of $P$.  Since $P$ has $N$ elements, $G_{=k}(P)$ is empty for $k > N$.

By Lemma \ref{inverseimage}, we can count $|Q(P)|$ in terms of $|G_{=k}(P)|$.

$$ |Q(P) | = \sum_{k=2}^N 2 {k \choose 2} | G_{=k}(P) |. $$

Let $G_k(P) \subset G$ be the set of $g \in G$ so that $| P \cap g P | \ge k$.  We see that $|G_{=k}(P)|
= |G_k(P)| - |G_{k+1}(P)|$.  Plugging this into the last equation and rearranging, we get the following.

\begin{equation} \label{quadformula}  |Q(P)| =  \sum_{k=2}^N 2 {k \choose 2} \left( |G_{k}(P)| - |G_{k+1}(P)| \right)  = \sum_{k=2}^N (2k-2) |G_k(P) |. \end{equation}

We will bound the number of partial symmetries as follows.

\begin{proposition} \label{partsymmbound} For any set $P \subset {\bf R^2}$ of $N$ points, and any $2 \le k \le N$,
the size of $G_k(P)$ is bounded as follows

$$ |G_k(P)| \lesssim N^3 k^{-2}. $$

\end{proposition}

When $P$ is a square grid, this estimate is sharp up to constant factors for all $2 \le k \le N$ (see the appendix).  
Plugging this bound into equation \ref{quadformula}, we get $|Q(P) | \lesssim N^3 \log N$, proving Proposition \ref{quadbound}.
This in turn implies our main theorem, Theorem \ref{main}.  So it suffices to prove Proposition \ref{partsymmbound}.

Next Elekes and Sharir related the sets $G_k(P)$ to an incidence problem involving certain curves in $G$.
For any points $p, q \in {\bf R^2}$, define the set $S_{pq}  \subset G$ given by

$$S_{pq} = \{  g \in G:  g(p) = q \}.$$

Each $S_{pq}$ is a smooth 1-dimensional
curve in the 3-dimensional Lie group $G$.  The sets $G_k(P)$ are closely related to the curves $S_{pq}$.

\begin{lemma} \label{gkincid} A rigid motion $g$ lies in $G_k(P)$ if and only if it lies in at least $k$ of the curves
$\{ S_{pq} \}_{p,q \in P}$.

\end{lemma} 

\begin{proof} First suppose that $g$ lies in $G_k(P)$.  By definition, $|P \cap gP | \ge k$.  Let $q_1, ..., q_k$ be
distinct points in $P \cap gP$.  Let $p_i = g^{-1}(q_i)$.  Since $q_i \in g P$, we see that $p_i$ lies in $P$.  Since
$g(p_i) = q_i$, we can say that $g$ lies in $S_{p_i q_i}$ for $i = 1, ..., k$.  Since the $q_i$ are all distinct, these
are $k$ distinct curves.

On the other hand, suppose that $g$ lies in the curves $S_{p_1 q_1}, ..., S_{p_k q_k}$, where we assume that the pairs $(p_1, q_1), ...
, (p_k, q_k)$ are all distinct.  We claim that $q_1, ..., q_k$ are distinct points.  To see this, suppose that $q_i = q_j$.  
Since $g$ is a bijection, we see that $p_i = g^{-1} (q_i) = g^{-1} (q_j) = p_j$, and this gives a contradiction.  But the points $q_1, ..., q_k$
all lie in $P \cap g P$.  \end{proof}

Bounding $G_k(P)$ is a problem of incidence geometry about the curves $\{ S_{pq} \}_{p,q \in P}$ in the group $G$.  By making
a careful change of coordinates, we can reduce this problem to an incidence problem for lines in ${\bf R^3}$.
(Our change of coordinates is slightly nicer than the one in \cite{ES}.  In the coordinates of \cite{ES}, the curves $\{ S_{pq} \}$ become
helices - a certain of class of degree 2 curves in ${\bf R^3}$.)

Let $G^{\prime}$ denote the open subset of the orientable rigid motion group $G$ given by rigid motions which are
not translations.  We can write $G$ as a disjoint union $G^{\prime} \cup G^{trans}$,  where $G^{trans}$ denotes the translations.
We then divide $G_k(P) = G_k^{\prime} \cup G_k^{trans}$.  Translations are a very special class of rigid motions, and
it is fairly easy to bound $|G_k^{trans}(P)| \lesssim N^3 k^{-2}$.  We carry out this minor step at the end of this Section.  
The main point is to bound $|G_k^{\prime}(P)|$.  To do this, we pick a nice set of coordinates $\rho: G' \rightarrow {\bf R^3}$.

Each element of $G^{\prime}$ has a unique fixed point $(x,y)$ and an angle $\theta$ of rotation about the
fixed point with $0 < \theta < 2\pi$.  We define the map

$$\rho :  G'  \longrightarrow  {\bf R}^3$$

by

$$\rho(x,y,\theta)=(x,y,\cot {\theta \over 2}).$$

\begin{proposition} \label{they're lines} Let $p=(p_x,p_y)$ and $q=(q_x,q_y)$ be points in ${\bf R}^2$. Then
with $\rho$ as above, the set $\rho(S_{pq} \cap G^{\prime})$ is a line in ${\bf R}^3$.
\end{proposition}

\begin{proof}  Noting that the fixed point of any transformation taking $p$ to $q$ must lie on the
perpindicular bisector of $p$ and $q$, the reader will easily verify that the set $\rho(S_{pq} \cap G^{\prime})$
can be parametrized as
\begin{equation} \label{Param} ({p_x + q_x \over 2}, {p_y + q_y \over 2},0)  +
t  (  {q_y -p_y \over 2}, {p_x-q_x \over 2}, 1).  \end{equation}

\end{proof}

For any $p,q \in {\bf R^2}$, let $L_{pq}$ denote the line $\rho(S_{pq} \cap G^{\prime})$.  The line $L_{pq}$ is parametrized by
equation \ref{Param}.  Let $\frak L$ be the set of lines $\{ L_{pq} \}_{p, q \in P}$.  By examining the
parametrization in equation \ref{Param}, it's easy to check that these are $N^2$ distinct lines.  If $g$ lies in
$G_k'(P)$, then $\rho(g)$ lies in at least $k$ lines of $\frak L$.  In the remainder of the paper, we will
study the set of lines $\frak L$ and estimate the number of points lying in $k$ lines.

We would like to prove that there are $\lesssim N^3 k^{-2}$ points that lie in at least $k$ lines of $\frak L$.  
Such an estimate does not hold for an arbitrary set of $N^2$ lines.  For example, if all the lines of $\frak L$
lie in a plane, then one may expect $\sim N^4$ points that lie in at least 2 lines.  This number of intersection
points is far too high.  There is another important example, which occurs when all the lines lie in a regulus.  
Recall that a regulus is a doubly-ruled surface, and each line from one ruling intersects all the lines from the
other ruling.  If $\frak L$ contained $N^2/2$ lines in each of the rulings, then we would have $\sim N^4$ points that lie in at least
2 lines.  Because of this example, we have to show that not too many lines of $\frak L$ lie in a plane or a regulus.

\begin{proposition} \label{genericity} 
No more than $N$ lines of $\frak L$ lie in a single plane.  No more than $O(N)$ lines of
$\frak L$ lie in a single regulus. \end{proposition}

\begin{proof}  For each $p \in P$, we consider the subset $\frak L_p \subset \frak L$ given by
$$\frak L_p = \{ L_{pq} \}_{q \in P}.$$

Notice that if $q \not= q'$, then $L_{pq}$ and $L_{p q'}$ cannot intersect.  So the lines of
$\frak L_p$ are disjoint.  From
equation \ref{Param}, it follows that the lines of $\frak L_p$ all have different directions.  So the lines
of $\frak L_p$ are pairwise skew, and no two of them lie in the same plane. Therefore, any plane
contains at most $N$ lines of $\frak L$.

The situation for reguli is more complicated because all $N$ lines of $\frak L_p$ may lie in
a single regulus.  But we will prove that this can only occur for at most two values of $p$.
To formulate this argument, we define $\frak L_p' := \{ L_{pq} \}_{q \in {\bf R^2}}$, so that $\frak L_p \subset \frak L_p'$.

\begin{lemma} \label{Lpruled} Suppose that a regulus $R$ contains at least five lines of
$\frak L_p'$.  Then all the lines in one ruling of $R$ lie in $\frak L_p'$.
\end{lemma}

Given this lemma, the rest of the proof of Proposition \ref{genericity} is straightforward.  
 If a regulus $R$ contains at least five lines of $\frak L_p$, then all the lines in one ruling
of $R$ lie in $\frak L_p'$.  But if $p_1 \not= p_2$, then $\frak L'_{p_1}$ and $\frak L'_{p_2}$ 
are disjoint, which we can check from the explicit formula
in equation \ref{Param}.   Since a regulus has only two rulings, there are at most two values of $p$ such that $R$ contains $\ge 5$ lines of $\frak L_p$.  
These two values of $p$ contribute $\le 2N$ lines of $\frak L$ in the surface $R$.  The other $N-2$ values of
$p$ contribute at most $4 (N-2)$ lines of $\frak L$ in the surface $R$.  
Therefore, the surface $R$ contains at most $2 N + 4 (N-2) \lesssim N$ lines of  $\frak L$. 

\begin{proof}[Proof of Lemma \ref{Lpruled}]  We fix the value of $p$.  We'll check below
that each point of ${\bf R^3}$ lies in exactly one line of $\frak L_p'$.  We will construct a non-vanishing vector field
$V = (V_1, V_2, V_3)$ on ${\bf R^3}$ tangent to the lines of $\frak L_p'$.  
Moreover, the coefficients $V_1, V_2$ and $V_3$ are all polynomials in $(x,y,z)$ of degree $\le 2$.
This construction is slightly tedious, but straightforward.  We postpone it to the end of the proof.

The regulus $R$ is defined by an irreducible polynomial $f$ of degree 2.  
Now suppose that a line $L_{pq}$ lies in $R$.  At each point $x \in L_{pq}$, the vector $V(x)$ points tangent to the line $L_{pq}$, and so
the directional derivative of $f$ in the direction $V(x)$ vanishes at the point $x$.  In other words the dot product 
$V \cdot \nabla f$ vanishes on the
line $L_{pq}$.  Since $f$ has degree 2, the dot product $V \cdot \nabla f$ is a degree 2 polynomial.

Suppose that $R$ contains five lines of $\frak L_p'$.  We know that $f$ vanishes on each line, and the previous paragraph
shows that $V \cdot \nabla f$ vanishes on each line.  By Bezout's theorem (see Lemma \ref{Bezoutlines}), $f$ and $V \cdot \nabla f$ 
must have a common factor.  Since $f$ is irreducible, we must have that $f$ divides $V \cdot \nabla f$.  In other words, $V \cdot \nabla
f$ vanishes on the surface $R$, and so $V$ is tangent to $R$ at every point of $R$.  If $x$ denotes any point in $R$, and we let $L$ be the line of $\frak L_p'$ containing $x$, then
we see that this line lies in $R$.  In this way, we get a ruling of $R$ consisting of lines from $\frak L_p'$.

It remains to define the vector field $V$.  We begin by checking that each point $(x,y,z)$ lies in exactly one line of
$\frak L_p'$.  By equation \ref{Param}, $(x,y,z)$ lies in $L_{pq}$ if and only if the following equation holds for some $t$.

$$ ({p_x + q_x \over 2}, {p_y + q_y \over 2},0)  +
t  (  {q_y -p_y \over 2}, {p_x-q_x \over 2}, 1) = (x, y, z). $$

Given $p$ and $(x,y,z)$, we can solve uniquely for $t$ and $(q_x, q_y)$.  First of all, we see that $t = z$.  Next we get
a matrix equation of the following form:

$$\left( \begin{array}{cc}
1 & z \\
-z & 1 \end{array} \right)   \left( \begin{array}{cc} q_x \\ q_y \end{array} \right) = a(x,y,z) . $$

\noindent In this equation, $a(x,y,z)$ is a vector whose entries are polynomials in $x,y, z$ of degree $\le 1$.  (The polynomials also depend on
$p$, but since $p$ is fixed, we suppress the dependence.)  Since the determinant of the matrix on the left-hand side is $1 + z^2 > 0$, we can solve this equation for $q_x$ and $q_y$.  The solution has
the form

\begin{equation} \label{qeq}  \left( \begin{array}{cc} q_x \\ q_y \end{array} \right) = (z^2 + 1)^{-1} b(x,y,z). \end{equation}

\noindent In this equation, $b(x,y,z)$ is a vector whose entries are polynomials in $x,y,z$ of degree $\le 2$.

The vector field $V(x,y,z)$ is $(z^2 + 1)  (  {q_y -p_y \over 2}, {p_x-q_x \over 2}, 1)$.  Recall that $p$ is fixed, and $q_x$ and $q_y$
can be expressed in terms of $(x,y,z)$ by the equation above.  By equation \ref{Param}, this
vector field is tangent to the line $L_{pq}$.  After multiplying out, the third entry of $V$ is $z^2 + 1$, so $V$ is non-vanishing.  Plugging
in equation \ref{qeq} for $q_x$ and $q_y$ and multiplying out, we see that the entries of $V(x,y,z)$ are polynomials
of degree $\le 2$.  \end{proof}

This concludes the proof of Proposition \ref{genericity}.  \end{proof}

We have now connected the distinct distance problem to the incidence geometry problem we mentioned in the introduction.  We know
that $\frak L$ consists of $N^2$ lines with $\lesssim N$ lines in any plane or regulus.  We now state our two results on incidence
geometry.

\begin{theorem}  \label{theorem1}  Let $\frak L$ be any set of $N^2$ lines in ${\bf R}^3$ for which no more than $N$ lie
in a common plane and no more than $O(N)$ lie in a common regulus.  Then the number of points of
intersection of two lines in $\frak L$ is $O(N^3)$.
\end{theorem}

\begin{theorem} \label{theorem2}  Let $\frak L$ be any set of $N^2$ lines in ${\bf R}^3$ for which no more than $N$ lie
in a common plane, and let $k$ be a number $3 \leq k  \leq N$.  Let $\frak S_k$ be the set of points where at least $k$ lines
meet.  Then

$$|\frak S_k| \lesssim N^3 k^{-2}.$$
\end{theorem}

Elekes and Sharir essentially conjectured these two theorems (Conjecture 1 in \cite{ES}).  (The difference is that they used
different coordinates, so their conjectures are about helices.)
In the case $k=3$, Theorem \ref{theorem2} was proven in \cite{EKS}.

Combining these theorems with the coordinates $\rho$ and Proposition \ref{genericity}, we get bounds for $|G_k'(P)|$.  Theorem 
\ref{theorem1} shows that $|G_2'(P)| \lesssim N^3$.  Theorem \ref{theorem2} shows that $|G_k'(P)| \lesssim N^3 k^{-2}$ for $3 \le k \le N$.

We now prove similar bounds for $|G_k^{trans}(P)|$.  These bounds are completely elementary

\begin{lemma} \label{transbound} Let $P$ be any set of $N$ points in ${\bf R}^2$.  The number of quadruples in $E^{-1}(G^{trans})$
is $\le N^3$.  Moreover, $|G^{trans}_k(P)| \lesssim N^3 k^{-2}$ for all $2 \le k \le N$.
\end{lemma}

\begin{proof} Suppose that $(p_1, p_2, p_3, p_4)$ is a distance quadruple in $E^{-1}(G^{trans})$.  By definition, there is a
{\it translation} $g$ so that $g(p_1) = p_3$ and $g(p_2) = p_4$.  Therefore, $p_3 - p_1 = p_4 - p_2$.  This equation allows us
to determine $p_4$ from $p_1, p_2, p_3$.  Hence there are $\le N^3$ quadruples in $E^{-1}(G^{trans})$.

By Proposition \ref{inverseimage} we see that

$$| E^{-1}(G^{trans}) | = \sum_{k=2}^N 2 {k \choose 2} |G^{trans}_{=k}(P)|. $$

Noting that $|G_k^{trans}(P)| = \sum_{l \ge k} |G_{=l}^{trans}(P)|$, we see that

$$ N^3 \ge | E^{-1}(G^{trans}) | \ge 2 {k \choose 2} |G_k^{trans}(P)| . $$

This inequality shows that $|G_k^{trans}(P)| \lesssim N^3 k^{-2}$ for all $2 \le k \le N$. \end{proof}

This substantially ends Section 2 of the paper.  To conclude, we give a summary and make some comments.

The new ingredients in this paper are Theorems \ref{theorem1} and \ref{theorem2}, which we prove in Sections 
3 and 4.  These theorems allow us to bound the partial symmetries of $P$ in $G'$: they imply that $|G'_k(P)| \lesssim
N^3 k^{-2}$ for all $2 \le k \le N$.  An elementary argument in Lemma \ref{transbound} shows the same estimates
for $|G^{trans}_k(P)|$.  Combining these, we see that $|G_k(P)| \lesssim N^3 k^{-2}$ for $2 \le k \le N$, proving
Proposition \ref{partsymmbound}.  Now the number of quadruples in $Q(P)$ is expressed in terms of $|G_k(P)|$
in equation \ref{quadformula}.  Plugging in our bound for $|G_k(P)|$, we get that $|Q(P)| \lesssim N^3 \log N$,
proving Proposition \ref{quadbound}.  Finally, the number of distinct distances is related to $|Q(P)|$ by Lemma \ref{CS}.  
Plugging in our bound for $|Q(P)|$, we see that $|d(P)| \gtrsim N (\log N)^{-1}$, proving our main theorem.

The group $G$ acts as a bridge connecting the original problem on distinct distances to the incidence
geometry of lines in ${\bf R^3}$.  The distance set $d(P)$ is related to the set of quadruples $Q(P)$ which is related
to the partial symmetries $G_k(P)$, which correspond to $k$-fold intersections of the lines in $\frak L$.  The group $G$ is a 
natural symmetry group for the problem of distinct distances, but this way of using the symmetry group is new and rather surprising.

Our estimates show that sets with few distinct distances must have many partial symmetries.  For example,
if $G_3(P)$ is empty, then our results show that $|Q(P)| \lesssim N^3$ and $|d(P)| \gtrsim N$.  Also, any set with
$|d(P)| \lesssim N (\log N)^{-1/2}$ must have a partial symmetry with $k \ge \exp (c \log^{1/2} N)$, for a universal constant
$c > 0$.  Any set with $|d(P)| \lesssim N (\log N)^{-1}$ must have a partial symmetry with $k \ge N^c$ for a universal $c > 0$.

\section{Flecnodes}

Our goal in this section is to prove Theorem \ref{theorem1}. We will do this by purely algebraic methods
following essentially the proof strategy of \cite{GK}. That is, we will show that an important subset of our
lines lies in the zero set of a fairly low degree polynomial $p$. What requires a new idea is the next step. We
need a polynomial $q$ derived from $p$ with similar degree on which the lines also vanish. With that information
we will apply a variant of Bezout's lemma.

\begin{lemma} \label{Bezoutlines}  Let $p(x,y,z)$ and $q(x,y,z)$ be polynomials on ${\bf R}^3$ of degrees
$m$ and $n$ respectively. If there is a set of $mn+1$ distinct lines simultaneously contained in the
zero set of $p$ and the zero set of $q$ then $p$ and $q$ have a common factor. \end{lemma}

Thus we will conclude that $p$ and the derived polynomial $q$ must have a common factor and we will arrive
at some geometrical conclusion from this based on the way that $q$ was derived. In the paper \cite{GK}, the derived
polynomials that we used were the gradient of $p$ and the algebraic version of the second fundamental form
of the surface given by $p=0$. These were good choices because when three or more lines were incident at each
point, we knew on geometric grounds that one or the other would vanish at each point, because the point would
be either critical or flat. However, here we are faced with points at which only two lines intersect, and so
we must make a more clever choice of the derived polynomial.

We begin with the definition of a flecnode. Given an algebraic surface in ${\bf R}^3$ given by the equation
$p(x,y,z)=0$ where p is a polynomial of degree $d$ at least 3, a flecnode is a point $(x,y,z)$ where a line agrees
with the surface to order three.
To find all such points, we might solve the system of equations:

$$p(x,y,z)=0; \quad  \nabla_v p(x,y,z)=0; \quad \nabla_v^2 p(x,y,z)=0; \quad \nabla_v^3 p(x,y,z)=0.$$

These are four equations for six unknowns,  $(x,y,z)$ and the components for the direction $v$. However the last three equations
are homogeneous in $v$ and may be viewed as three equations in five unknowns (and the whole system as 4 equations in 5 unknowns.)
We may reduce  the last three equations
to a single equation in three unknowns $(x,y,z)$. We write the reduced equation as
$$\operatorname{Fl}(p) (x,y,z) = 0.$$
The polynomial $\operatorname{Fl}(p)$ is of degree $11d-24$. It is called the flecnode polynomial of $p$ and vanishes at any
flecnode of any level set of $p$. (See \cite{Salm}  Art. 588 pages 277-78.)

The term flecnode was apparently first coined by Cayley. The polynomial 
$\operatorname{Fl}(p)$ was discovered by the
Rev. George Salmon, but its most important property to us was communicated to him by Cayley.

\begin{proposition} \label{flecnode} The surface $p=0$ is ruled if and only if 
$\operatorname{Fl}(p)$ is everywhere vanishing
on it. \end{proposition}

An algebraic surface (in ${\bf R}^3$) is ruled if it contains a line passing through every point. The set of all
lines contained in an algebraic surface (of some degree N) is an algebraic set of lines. (This
is because a line is contained in the surface if and only if it is contained to
order $N+1$ at one of its points. So a line is contained in the surface if and only if $N+1$ polynomial equations
in the parameters of the line are satisfied.) The set of lines contained in a surface may have two
dimensional components, one-dimensional components and zero-dimensional components. It is easy to see
that an algebraic surface in ${\bf R}^3$ contains a two dimensional set of lines only if it has a plane as a factor.
(The way to see this is to find a regular point of the surface with an infinite number of lines going through it. Then
the surface must contain the tangent plane to this point.) 
 Thus an algebraic surface which is ruled and plane-free will
contain both a 1-dimensional set of lines (the generators) and possibly a 0-dimensional set
of lines. A detailed classical treatment of ruled surfaces is given in \cite{Salm}  Chapter XIII
Part 3.

One important example of a ruled surface is a regulus.  A regulus is actually doubly-ruled: 
every point in the regulus lies in two lines in the regulus.  A ruled surface is called singly-ruled if
a generic point in the surface lies in only one line in the surface.  (Some points in a singly-ruled
surface may lie in two lines.)  Except for reguli and planes, every irreducible ruled
surface (in ${\bf R}^3$) is singly-ruled.  (See the explanation in Section 1).

One direction of Proposition \ref{flecnode} is obvious. If the surface is ruled, there is a line contained in the surface
at every point. If the line is contained in the surface, it certainly agrees to order 3. The reverse direction is
more computational. It is indicated in a footnote to \cite{Salm}  Art. 588 page 278. One sees that setting $\operatorname{Fl}(p)=0$
is a way of rewriting a differential equation on $p$ which implies ruledness.  Proposition \ref{flecnode} was
used in a famous paper of Segre \cite{Seg}. For a generalization to manifolds in higher dimensions see \cite{Land}.

An immediate corollary of the proposition is

\begin{corollary} \label{flecbound} Let $p=0$ be a degree $d$ hypersurface in ${\bf R}^d$. Suppose that
the surface contains more than $11d^2 -24d$ lines. Then $p$ has a ruled factor. \end{corollary}

\begin{proof} By lemma \ref{Bezoutlines}, since both $p$ and $\operatorname{Fl}(p)$ vanish on the same set of more than
$11d^2-24d$ lines, they must have a common factor $q$. Since $q$ is a factor of $p$
and $\operatorname{Fl}(p)$ vanishes on the surface $q=0$, it
must be that at every regular point of the surface $q=0$, there is a line which meets the surface
to order 3. Thus $\operatorname{Fl}(q)=0$ which implies by Proposition \ref{flecnode}
that $q$ is ruled.
 \end{proof}

Now we would like to consider ruled surfaces of degree less than $N$. Thus our surfaces are the sets
$$p(x,y,z)=0$$
for a polynomial $p$ (which we may choose square free) of degree less than $N$. We may uniquely factorize the
polynomial into irreducibles:
$$p=p_1 p_2 \dots p_m.$$
We say that $p$ is plane-free and regulus-free if none of the zero sets of the factors is a plane or a regulus.
Thus if $p$ is plane-free and regulus-free, the zero-set of each of the factors is an irreducible algebraic
singly-ruled surface. We now state the main geometrical lemma for proving Theorem \ref{theorem1}.

\begin{lemma} \label{ruled}  Let $p$ be a polynomial of degree less than $N$ so that $p=0$ is ruled
and so that $p$ is plane-free and regulus-free. Let $\frak L_1$ be a set of lines contained in the
surface $p=0$ with $| \frak L_1| \lesssim N^2$.  Let $Q_1$ be the set of points of intersection
of lines in $\frak L_1$.  Then
$$|Q_1| \lesssim N^3.$$
\end{lemma}

Before we begin in earnest the proof of Lemma \ref{ruled}, we will nail down a few delicate points of the geometry
of irreducible singly-ruled surfaces.

We let $p(x,y,z)$ be an irreducible polynomial so that $p(x,y,z)=0$ is a ruled surface which is not a plane or a regulus.
In other words, the surface $S=\{(x,y,z): p(x,y,z)=0\}$ is irreducible and singly-ruled. We say that a point $(x_0,y_0,z_0) \in S$ 
 is an {\it exceptional point} of the surface, if it lies on infinitely many lines contained in the surface. We say that a line
$l$ contained in $S$ is an {\it exceptional line} of the surface if there are infinitely many lines in $S$ which intersect $l$
at non-exceptional points. We prove a structural lemma about exceptional points and exceptional lines of irreducible singly-ruled surfaces.

\begin{lemma} \label{exceptional} Let $p(x,y,z)$ be an irreducible polynomial. Let $S=\{(x,y,z): p(x,y,z)=0 \}$ be an irreducible
surface which is neither a plane nor a regulus.
\begin{enumerate}
\item  Let $(x_0,y_0,z_0)$ be an exceptional point of $S$. Then every other point $(x,y,z)$ of $S$ is on a line $l$ which is
contained in $S$ and which contains the point $(x_0,y_0,z_0)$.
\item  Let $l$ be an exceptional line of $S$. Then there is an algebraic curve $C$ so that every point of $S$ not lying on $C$ is contained
in a line contained in $S$ and intersecting $l$.
\end{enumerate}
\end{lemma}

We proceed to give an elementary proof of Lemma \ref{exceptional}:

\begin{proof} 
To prove the first part, we observe that by a change of coordinates we can move $(x_0,y_0,z_0)$ to the origin.
We let $Q$ be the set of points $q$ different from the origin so that the line from $q$ to the origin is contained in $S$.
We observe that $Q$ is
the intersection of an algebraic set with the complement of the origin. That is, there is a finite set of polynomials $E$ so that
a point $q$ different from the origin lies in $Q$ if and only if each polynomial in $E$ vanishes at $q$. This is because if $d$ is
the degree of $p$, to test whether $q \in Q$, we need only check that the line containing $q$ and the origin is tangent to $S$
to degree $d+1$ at $q$. Now by assumption, the zero set of each polynomial in $E$ contains the union of infinitely many lines 
contained in $S$. Thus by Lemma \ref{Bezoutlines} and by the irreducibility of $p$, it must be that each polynomial in $E$ has $p$
as a factor.  Therefore $Q$ is all of $S$ except the origin. We have proved the first part.

Now to prove the second part, we observe that by a change of coordinates, we may choose $l$ to be the coordinate line $y=0; z=0$.
We let $Q$ be the set of points $q$ not on $l$ so that there is a line from $q$ to a non-exceptional point of $l$ which is contained in
$S$. We would like to claim that $Q$ is the intersection of an algebraic set with the complement of an algebraic curve. If we
are able to show this, we will prove the second claim in the same way that we proved the first. To do this, 
for points $(x,y,z)$ on $S$ outside of an algebraic curve, we will identify the point at which the line containing $(x,y,z)$ intersects $l$.

Consider a point $(x,y,z)$ on $S$ for which ${\partial p \over \partial x} (x,y,z) \neq 0$. In particular, the point $(x,y,z)$ is
a regular point of $S$. Since ${\partial p \over \partial x} (x,y,z) \neq 0$, there is a unique point $(x^{\prime},0,0)$ of $l$
which lies in the tangent plane to $S$ at the point $(x,y,z)$. In fact, we can solve for $x^{\prime}$ as a rational function of
$(x,y,z)$ with only the polynomial ${\partial p \over \partial x}$ in the denominator. Thus we can find a set $E$ of rational functions
having only powers of ${\partial p \over \partial x}$ in their denominators, so that for any $(x,y,z)$ at which
${\partial p \over \partial x}$ does not vanish, we have that $(x,y,z) \in Q$ if and only if, every function in $E$ vanishes on $(x,y,z)$.

In order for the previous paragraph to be useful to us, we need to know that ${\partial p \over \partial x}$ does not vanish
identically on $S$. Suppose that it did. Since ${\partial p \over \partial x}$ is of lower degree than $p$ and $p$ is irreducible,
it must be that ${\partial p \over \partial x} $ vanishes identically as a polynomial so that $p$ depends only on $y$ and $z$. In this
case, since $S$ contains $l$ and it contains a line $l^{\prime}$ intersecting $l$, it must contain all translates of $l^{\prime}$ in the 
$x$-direction. Thus it contains a plane which is a contradiction.

Thus, we let $C$ be the algebraic curve where both $p$ and ${\partial p \over \partial x}$ vanish. Away from $C$, there is a 
finite set of polynomials $F$ (which we obtain from $E$ by multiplying by a large enough power of ${\partial p \over \partial x} $)
so that a point $(x,y,z)$ of $S$ outside of $C$  is in $Q$  if and only if each polynomial in $F$ vanishes at $(x,y,z)$. Since
we know that $p$ is irreducible and $Q$ contains an infinite number of lines, it must be that each polynomial in $F$ has $p$
as a factor. Thus every point of $S$ which is outside of $C$ lies in $Q$ which was to be shown.

\end{proof}

Now that we have established our structural result, Lemma \ref{exceptional}, we may use it to obtain a corollary
giving quantitative bounds on the number of exceptional points and lines.  

\begin{corollary} \label{takingexception} Let $p(x,y,z)$ be an irreducible polynomial. Let $S=\{(x,y,z): p(x,y,z)=0 \}$ be an irreducible
surface which is neither a plane nor a regulus. Then $S$ has at most one exceptional point and at most two exceptional lines.
\end{corollary}

We now prove Corollary \ref{takingexception}

\begin{proof} Let $(x_0,y_0,z_0)$ and $(x_1,y_1,z_1)$ be distinct exceptional points of $S$. Since $S$ is singly-ruled, the generic point
of $S$ is contained in only a single line $l$ contained in $S$. Thus by Lemma \ref{exceptional}, if the point is different 
from $(x_0,y_0,z_0)$ and $(x_1,y_1,z_1)$, this line $l$ must contain both $(x_0,y_0,z_0)$ and $(x_1,y_1,z_1)$ . But there is
only one such line, and that is a contradiction.

Now let $l_1,l_2,l_3$ be exceptional lines of $S$. There are curves $C_1,C_2,$ and $C_3$ so that the generic point in the complement of
$C_1,C_2,$ and $C_3$ lies on only one line contained in $S$ and this line must intersect each of $l_1,l_2,$ and $l_3$. Thus there
are infinitely many lines contained in $S$ which intersect each of $l_1,l_2,$ and $l_3$. (Moreover, since the lines are exceptional, there
must be an infinite set of lines which intersect the three away from the possible three points of intersection of any two of $l_1,l_2,$ and $l_3$.)
If the any two of the three lines are coplanar, this means there is an infinite set of lines contained in $S$ which lie in one plane.
This contradicts the irreducibility and nonplanarity of $S$. If contrariwise, the three lines $l_1,l_2,$ and $l_3$
are pairwise skew, then the set of all lines which intersect all three are one ruling of a regulus. In this case, $S$ contains
infinitely many lines of a regulus which contradicts the fact that $S$ is irreducible and not a regulus.

\end{proof}

For context, we remark that an irreducible singly-ruled surface with an exceptional point is often referred to as a cone
and the exceptional point is referred to as the cone point.  Irreducible ruled surfaces with two exceptional lines do exist:
one way of constructing a ruled surface with two exceptional lines is to
choose a curve in the two-dimensional set of lines which intersect a pair of skew lines.

At last, we may begin the proof of Lemma \ref{ruled}.

\begin{proof}
We say that a point $(x,y,z)$ is exceptional for the surface $p=0$, if it is exceptional
for $p_j=0$ where $p_j$ is one of the irreducible factors of $p$. We say that a line $l$
is exceptional for the surface $p=0$ if it is exceptional for $p_j=0$ where $p_j$ is
one of the irreducible factors of $p$. Thus, in light of Corollary \ref{takingexception},
there are no more than $N$ exceptional points and $2N$ exceptional lines for $p=0$. 
Thus
there are $\lesssim N^3$ intersections between exceptional lines and lines of $\frak L_1$.
Thus to prove the lemma, we need only consider intersections between nonexceptional lines
of $\frak L_1$ at nonexceptional points.

We note that any line contained in a ruled surface which is not a generator must be an 
exceptional line since each point of the line will have a generator going through it.  (The definition
of a ruled surface is that every point lies in a line in the surface.  Since there are only finitely many
non-generators, almost every point must lie in a generator.  But in fact every point lies in a generator
by a limiting argument.  Let $q$ be a point in the ruled surface and let $q_i$ be a sequence of points
that converge to $q$ with $q_i$ lying in a generator $l_i$.  By taking a subsequence, we can arrange
that the directions of the $l_i$ converge, and so the lines $l_i$ converge to a limit line $l$ which contains
$q$ and lies in the surface.  This line is a limit of generators, and so it is a generator.)

Let $l$ be a non-exceptional line in the ruled surface.  In particular $l$ is a generator.  We claim
that there are at most $N-1$ non-exceptional points in $l$ where $l$ intersects another non-exceptional
line in the ruled surface.  This claim implies that there are at most $(N-1) N^2$ non-exceptional points where two
non-exceptional lines intersect, proving the bound we want.

To prove the claim, we repeat
an argument found in \cite{Salm} Art 485 pages 88-89.  Choose a plane $\pi$
through the generator $l$.  The plane intersects the surface in a curve of degree $N$. One component is the generator
itself. The other component is an algebraic curve $c$ of degree $N-1$.  There are at most $N-1$ points of intersection
between $l$ and $c$.  Suppose that $l'$ is another non-exceptional line and that $l'$ intersects $l$ at a non-exceptional
point $q$.  It suffices to prove that $q$ lies in the curve $c$.  Since $l'$ is a generator, it lies in a continuous 1-parameter family
of other generators.  Consider a small open set of generators around $l'$.  These generators intersect the plane $\pi$.  So each of them intersects either $l$ or $c$.  Since $q$ is non-exceptional, only finitely many of them intersect $q$.  Since there are only
finitely many exceptional points, we can arrange that each generator in our small open set intersects $\pi$ in a non-exceptional point.
Since $l$ is non-exceptional, only finitely many of our generators can intersect
$l$.  Therefore, almost all of our generators must intersect $c$.  This is only possible if $q$ lies in $c$.
 \end{proof}

Now we are ready to begin the proof of Theorem \ref{theorem1}. We assume we have a set $\frak L$ of at most
$N^2$ lines for which no more than $N$ lie in a plane and no more than $N$ lie in a regulus. We  suppose, by
way of contradiction, that for $Q$, a positive real number sufficiently large, there are $Q N^3$ points of
intersection of lines of $\frak L$ and we assume that this is an optimal example, so that for no $M < N$ do we
have a set of $M^2$ lines so that no more than $M$ lie in a plane and no more than $M$ lie in a regulus
is it the case that there are more than $QM^3$ intersections. ($N$ need not be an integer.)

We now apply a degree reduction argument similar to the one in \cite{GK}. We let $\frak L^{\prime}$
be the subset of $\frak L$ consisting of lines which intersect other lines of $\frak L$ in at least ${QN \over 10}$ 
different points.
The lines not in $\frak L^{\prime}$ participate in at most ${QN^3 \over 10}$ points of intersection. Thus there
are at least ${9 Q N^3 \over 10}$ points of intersection between lines of  $\frak L^{\prime}$. We define a number
$\alpha$ with $0 < \alpha \le 1$ so that $\frak L^{\prime}$
has $\alpha N^2$ lines.

Now we select a random subset $\frak L^{\prime \prime}$ of the lines of $\frak L^{\prime}$ choosing lines
independently with probability ${100 \over Q}$. With positive probability, there will be no more than
${200 \alpha N^2 \over Q}$ lines in $\frak L^{\prime \prime}$ and each line of $\frak L^{\prime}$ will
intersect lines of $\frak L^{\prime \prime}$ in at least $N$ different points.  Now pick ${R \sqrt{\alpha} N \over \sqrt{Q}}$
points on each line of $\frak L^{\prime \prime}$. ($R$ is a constant which is sufficiently large but universal.)
Call the set of all of the points $\frak S$. There are
$O({R \alpha^{{3 \over 2}} N^3 \over Q^{{3 \over 2}}})$ points in $\frak S$, so we may find a polynomial $p$ of
degree $O({ R^{{1 \over 3}}\alpha^{{1 \over 2}} N \over Q^{{1 \over 2}}})$ which vanishes on every point of $\frak S$.
With $R$ sufficiently large, $p$ must vanish identically on every line of  $\frak L^{\prime \prime}$.
Since each line of $\frak L^{\prime}$ meets $\frak L^{\prime \prime}$ at $N$ different points, it must be that
$p$ vanishes identically on each line of $\frak L^{\prime}$. Thus ends the degree reduction argument and we
will now study the relatively low degree polynomial $p$.

We may factor $p=p_1 p_2$ where $p_1$ is the product of the ruled irreducible factors of $p$ and $p_2$ is the
product of unruled irreducible factors of $p$. Each of $p_1$ and $p_2$ is of degree
$O({\alpha^{{1 \over 2}} N \over Q^{{1 \over 2}}})$. (We have suppressed the $R$ dependence since $R$ is universal.)
We break up the set of lines of $\frak L^{\prime}$ into the disjoint subsets $\frak L_1$ consisting of those
lines in the zero set of $p_1$ and $\frak L_2$ consisting of all the other lines in $\frak L^{\prime}$.

There are no more than $O(N^3)$ points of intersection between lines of $\frak L_1$ and $\frak L_2$ since each line
of $\frak L_2$ contains no more than $O({\alpha^{{1 \over 2}} N \over Q^{{1 \over 2}}})$ points where
$p_1$ is zero.  Thus we are left with two (not mutually-exclusive) cases which cover all possibilities. There
are either ${3 Q N^3 \over 10}$ points of intersection between lines of $\frak L_1$  or there are
${3 Q N^3 \over 10}$ points of intersection between lines of $\frak L_2$. We will handle these separately.

Suppose there are ${3 Q N^3 \over 10}$ intersections between lines of $\frak L_1$. We factor $p_1=p_3 p_4$
where $p_3$ is plane-free and regulus-free and $p_4$ is a product of planes and reguli. We break
$\frak L_1$ into disjoint sets $\frak L_3$ and $\frak L_4$, with $\frak L_3$ consisting of
lines in the zero set of $p_3$  and $\frak L_4$ consisting of all other lines of $\frak L_1$. As before
there $O(N^3)$ points of  intersection between lines of $\frak L_3$ and $\frak L_4$ since lines of $\frak L_4$
are not in the zero set of $p_3$. Moreover there are at most $O(N^3)$ points of intersection between
lines of $\frak L_4$ because they lie in at most $N$ planes and reguli each containing at most $N$ lines.
(We just see that each line has at most $O(N)$ intersections with planes and reguli it is not contained in
and there are at most $O(N^2)$ points of intersection between lines internal to each plane and regulus.)
However there cannot be more than $O(N^3)$ points of intersection between lines of $\frak L_3$ by applying the
key lemma \ref{ruled}. (Here we used that $p_3$ is plane-free and regulus-free.)

Thus we must be in the second case, where many of the points of intersection are between lines of $\frak L_2$,
all of which lie in the zero set of $p_2$ which is totally unruled. Recall that $p_2$ is of degree
$O({\alpha^{{1 \over 2}} N \over Q^{{1 \over 2}}})$. Thus by Corollary \ref{flecbound}, its zero-set
contains no more than $O({\alpha N^2 \over Q})$ lines. We would like to now invoke the fact that the example we
started with was optimal and reach a contradiction. But we can't quite do that. Our set
$\frak L_2$ has $\beta N^2$ lines with $\beta=O({\alpha \over Q})$  and we only know that there are no more than
$N$ lines in any plane or regulus, whereas we need to know that there are no more than $\sqrt{\beta} N$ lines.
If this is the case we are done. If not we construct a subset $\frak L_5$ as follows. If there is a plane
or regulus containing more than $\sqrt{\beta} N$ lines of $\frak L_2$, we put those lines in $\frak L_5$ and
remove them from $\frak L_2$. We repeat as needed labelling the remaining lines $\frak L_6$. Since we removed
$O(N)$ planes and reguli, there are $O(N^3)$ points of intersection between lines of $\frak L_5$. Since
no lines of $\frak L_6$ belong to any plane or regulus of $\frak L_5$ there are fewer than $O(N^3)$ points
of intersection between lines of $\frak L_5$ and $\frak L_6$.  Now we apply optimality of our original
example to rule out more than $O({N^3 \over Q^{{1 \over 2}}})$ points of intersection between lines of
$\frak L_6$. Thus we have reached a contradiction.

\section{Cell decompositions}

In this section, we construct a new type of cell decomposition of ${\bf R^n}$, where the walls of the cells
are the zero set of a polynomial.  We use this type of cell decomposition to prove an incidence theorem for lines
in ${\bf R^3}$ when not too many lines lie in a plane.  The cell decomposition is described in the following
theorem.

\begin{theorem} \label{polycell} If $\frak S$ is a set of $S$ points in ${\bf R^n}$ and $J \ge 1$ is an integer,
then there is a polynomial surface $Z$ of degree $d \lesssim 2^{J/n}$ with the following property.  The complement ${\bf R^n}
\setminus Z$ is the union of $2^J$ open cells $O_i$, and each cell contains $\le 2^{-J} S $ points of $\frak S$.
\end{theorem}

Remark: Some or all of the points of $\frak S$ may lie inside the surface $Z$.  Recall that $Z$ is not part of any of the open
sets $O_i$.  So there are two extreme cases in Theorem \ref{polycell}.  In one extreme, all the points of $\frak S$ lie in the open
cells $O_i$, and there are exactly $2^{-J} S$ points in each cell.  In the other extreme, all the points of $\frak S$ lie in the surface
$Z$. When the points all lie in $Z$, the theorem does not give any information about where in $Z$ they lie.

The proof of Theorem \ref{polycell} is based on the polynomial ham sandwich theorem of Stone and Tukey \cite{ST}.  
For context, we first recall the original ham sandwich theorem.

\begin{theorem} \label{origham} (Ham sandwich theorem) If $U_1, ..., U_n \subset {\bf R^n}$ are finite volume
open sets, then there is a hyperplane which bisects each set $U_i$.
\end{theorem}

The ham sandwich theorem was proven in the case $n=3$ by Banach in the late 30's, using the Borsuk-Ulam theorem.
In 1942, Stone and Tukey generalized Banach's proof to all dimensions.  They also observed that the same
argument applies to many other situations.  In particular, they proved the following polynomial version of the ham 
sandwich theorem.  

We say that an algebraic hypersurface $p(x_1,\dots,x_n)=0$ bisects a finite volume open set $U$ if

$$ Vol ( U \cap \{p < 0\} ) = Vol (U \cap \{p > 0 \}) = (1/2) Vol (U). $$

\begin{theorem} \label{ham} (Stone-Tukey, \cite{ST}) For any degree $d \ge 1$, the following holds.  
Let $U_1,\dots, U_M$ be any finite volume open sets in ${\bf R}^n$, 
with $M={n+d \choose n} - 1$. Then there is a real algebraic hypersurface of degree at most $d$
that bisects each $U_i$. \end{theorem}

(For a recent exposition of the proof, see \cite{G}.)

We now adapt Theorem \ref{ham} to finite sets of points. Instead of open sets $U_i$, we will have finite
sets $S_i$.  We say that
a polynomial $p$ bisects a finite set $S$ if at most half the points in $S$ are in $\{ p > 0 \}$ and at most
half the points in $S$ are in $\{ p < 0\}$. Note that $p$ may vanish on some or all of the points of $S$.

\begin{corollary} \label{spam} Let $S_1,\dots, S_M$ be finite sets of points in  ${\bf R}^n$
 with $M={n+d \choose n} - 1$. Then there is a real algebraic hypersurface of degree at most $d$
that bisects each $S_i$. \end{corollary}

\begin{proof} For each $\delta > 0$, define $U_{i, \delta}$ to be the union of $\delta$-balls centered at
the points of $S_i$.  By the polynomial ham sandwich theorem, Theorem \ref{ham}, we can find a non-zero
polynomial $p_\delta$ of degree $\le d$ that bisects each set $U_{i, \delta}$.  

We want to take a limit of the polynomials $p_\delta$ as $\delta \rightarrow 0$.  To help make this work,
we pick a norm $\| \|$ on the space of polynomials of degree $\le d$.  Any norm will do - to be definite,
let $\| p \|$ denote the maximal absolute value of the coefficients of $p$.  By scaling $p_\delta$, we can
assume that $\| p_\delta \| = 1$ for all $\delta$.  Now we can find a sequence $\delta_m \rightarrow 0$
so that $p_{\delta_m}$ converges in the space of degree $\le d$ polynomials.  We let $p$ be the limit
polynomial and observe that $\| p \| =1$.  In particular, $p$ is not 0.  Since the coefficients of $p_{\delta_m}$
converge to the coefficients of $p$, it's easy to check that $p_\delta$ converges to $p$ uniformly on compact
sets.

We claim that $p$ bisects each set $S_i$.  We prove the claim by contradiction.  Suppose instead that
$p > 0$ on more than half of the points of $S_i$.  (The case $p < 0$ is similar.)  Let $S_i^+ \subset S_i$ denote the
set of points of $S_i$ where $p > 0$.  By choosing $\epsilon$
sufficiently small, we can assume that $p > \epsilon$ on the $\epsilon$-ball around each point 
of $S_i^+$.  Also, we can choose $\epsilon$ small enough that the $\epsilon$-balls around the points of $S_i$
are disjoint.  Since $p_{\delta_m}$ converges to $p$ uniformly on compact sets, we can find $m$ large enough
that $p_{\delta_m} > 0$ on the $\epsilon$-ball around each point of $S_i^{+}$.  By making $m$ large, we can also
arrange that $\delta_m < \epsilon$.  Therefore, $p_{\delta_m} > 0$ on the $\delta_m$-ball around each point of $S_i^+$.
But then $p_{\delta_m} > 0$ on more than half of $U_{i, \delta_m}$.  This contradiction proves that $p$ bisects $S_i$.
\end{proof}

Using this finite polynomial ham sandwich theorem, we can quickly prove Theorem \ref{polycell}.

\begin{proof}[Proof of Theorem \ref{polycell}]  We do the construction in $J$ steps.  In the first
step, we pick a linear polynomial $p_1$ that bisects $\frak S$.  We let $\frak S^+$ and $\frak S^-$
be the sets where $p_1$ is positive and negative, respectively.  In the second step, we find a polynomial
$p_2$ that bisects $\frak S^+$ and $\frak S^-$.  And so on.  At each new step, we use Corollary \ref{spam}
to bisect the sets from the previous step.

We now describe the inductive procedure a little more precisely.  At the end of step $j$, we have defined
$j$ polynomials $p_1, ..., p_j$.  We define $2^j$ subsets of $\frak S$ by looking at the points where the polynomials
$p_1, ..., p_j$ have specified signs.  Then we use Corollary \ref{spam} to bisect each of these $2^j$ sets.  It follows 
by induction that each subset contains $\le 2^{-j} S$ points.

Finally, we let $p$ be the product $p_1 ... p_J$, and we let $Z$ denote the zero set of $p$.

First we estimate the degree of $p$.  By Corollary \ref{spam}, the degree of $p_j$ is $\lesssim 2^{j/n}$.  Hence
the degree of $p$ is $d \lesssim \sum_{j=1}^J 2^{j/n} \lesssim 2^{J/n}$.

Now we define the $2^J$ open sets $O_i$ as the sets where the polynomials $p_1, ..., p_J$ have specified signs.
For example, one of the sets $O_i$ is defined by the inequalities $p_1(x) > 0, p_2(x) < 0, p_3 (x) > 0, ..., p_J(x) > 0$.  
The sets $O_i$ are open and disjoint.  Their union is exactly the complement of $Z$.  As we saw above, the number of points in 
$\frak S \cap O_i$ is at most $2^{-J} S$. \end{proof}

Using this type of cell decomposition, we will prove an estimate for incidences of lines when not too many lines lie in a plane.

\begin{theorem} \label{incidence}  Let $k \ge 3$.  Let $\frak L$ be a set of $L$ lines in ${\bf R^3}$ with
at most $B$ lines in any plane.  Let $\frak S$ be the set of points in ${\bf R^3}$ intersecting
at least $k$ lines of $\frak L$.  Then the following inequality holds:

$$ | \frak S | \le C [ L^{3/2} k^{-2} + L B k^{-3} + L k^{-1}]. $$

\end{theorem}

Theorem \ref{incidence} implies Theorem \ref{theorem2} by setting $L = N^2$ and $B = N$.

This theorem is sharp up to constant factors in a number of cases.  These examples help to give a
sense of the right-hand side.

{\bf Example 1.} Choose $L/k$ points.  Let $\frak L$ consist of $k$ lines through each point.
The set $\frak L$ has a k-fold incidence at each of the $L/k$ points.  (We can also arrange that no three lines lie in a plane.)

{\bf Example 2.} Choose $L/B$ planes.  Put $B$ lines in each of the planes.  The $B$ lines in each
plane can be arranged to create $B^2 k^{-3}$ k-fold incidences.  (See the examples in \cite{SzT}.)  This set of lines has a total of $L B k^{-3}$
k-fold incidences.

{\bf Example 3.} Let $G_0$ denote the integer lattice $\{ (a, b, 0) \}$ with $1 \le a, b \le L^{1/4}$.  Let
$G_1$ denote the integer lattice $\{ (a, b, 1) \}$ with $1 \le a,b \le L^{1/4}$.  Let $\frak L$ denote all the lines
from a point of $G_0$ to a point of $G_1$.  The horizontal planes $z=0$ and $z=1$ do not contain any lines of
$\frak L$.  Any other plane contains at most $L^{1/4}$ points of each $G_i$, and so
at most $L^{1/2}$ lines of $\frak L$.  We will prove in the appendix that there are $\sim L^{3/2} k^{-2}$
points that lie in $\ge k$ lines of $\frak L$ for each $k$ in the range $2 \le k \le L^{1/2}/400$.

For context, we should compare Theorem \ref{incidence} to the Szemer\'edi-Trotter theorem, which holds in all
dimensions as we now recall.

\begin{theorem} \label{sthigh} If $\frak L$ is a set of $L$ lines in ${\bf R^n}$, and $\frak S$ denotes the set of points lying
in at least $k$ lines of $\frak L$, then

$$ | \frak S | \lesssim L^2 k^{-3} + L k^{-1}. $$
\end{theorem}

\noindent The higher-dimensional case follows easily from the two-dimensional case by taking a generic projection from
${\bf R^n}$ to ${\bf R^2}$.  The set of lines $\frak L$ will project to $L$ distinct lines in ${\bf R^2}$, 
and the points of $\frak S$ project to distinct points in ${\bf R^2}$.

Theorem \ref{incidence} is a refinement of Theorem \ref{sthigh}.  When $B = L$, Theorem \ref{incidence}
is Theorem \ref{sthigh}.  Theorem \ref{incidence} tells us how much we can improve the Szemer\'edi-Trotter
theorem if we know in addition that not too many lines lie in a plane.

We will use Theorem \ref{sthigh} in our proof.  Recently, in \cite{KMS}, Kaplan, Matou\u{s}ek, and Sharir gave
a new proof of the Szemer\'edi-Trotter theorem using polynomial cell decompositions.

Now we turn to the proof of Theorem \ref{incidence}.
An important special case is the uniform case where each point
has $\sim k$ lines through it and each line contains about the same number of points.
We will first prove the theorem under some uniformity hypotheses.

\begin{proposition} \label{uniform} Let $k \ge 3$.  Let $\frak L$ be a set of $L$ lines in ${\bf R^3}$ with at most
$B$ lines in any plane.  Let $\frak S$ be a set of $S$ points in ${\bf R^3}$ so that each point
intersects between $k$ and $2k$ lines of $\frak L$.

Also, we assume that there are $\ge \frac{1}{100} L$ lines in $\frak L$ which each contain
$\ge \frac{1}{100} S k L^{-1}$ points of $\frak S$.

Then $S \le C [ L^{3/2} k^{-2} + L B k^{-3} + L k^{-1} ]$.

\end{proposition}

The second paragraph of Proposition \ref{uniform} is a uniformity assumption about the lines.  Note that there are $\sim S k$ total incidences
between lines of $\frak L$ and points of $\frak S$.  Therefore, an average line of $\frak L$ contains $\sim S k L^{-1}$
points of $\frak S$.  We assume here that there are many lines that are about average.  Proposition \ref{uniform} is the main
part of the proof of Theorem \ref{incidence}.  The general case reduces to this special case by easy inductive arguments.

\begin{proof} We begin by outlining our strategy.   We suppose that

\begin{equation} \label{bigs} S \ge A L^{3/2} k^{-2} + L k^{-1}. \end{equation}  

In this equation, $A$ represents a large constant that we will choose below.  Assuming
\ref{bigs}, we need to show that many lines of $\frak L$ lie in a plane.  In particular,
we will find a plane that contains $\gtrsim S L^{-1} k^3$ lines of $\frak L$.  This means
that $B \gtrsim S L^{-1} k^3$, and hence $S \lesssim B L k^{-3}$, and we will be done.

Let us outline how we find the plane.  First we prove that a definite fraction of the lines of $\frak L$
lie in an algebraic surface $Z$ of degree $\lesssim L^2 S^{-1} k^{-3}$.  Second we prove that
this variety $Z$ contains some planes, and that a definite fraction of the lines of $\frak L$
lie in the planes.  Since there are at most $d$ planes, one plane must contain $\gtrsim
L/d$ lines.  Because $d \lesssim L^2 S^{-1} k^{-3}$, this plane contains $\gtrsim S L^{-1} k^{3}$
lines, which is what we wanted to prove. 

Our bound for the degree $d$ is sharp up to a constant factor because of Example 2 above.  In this example, the lines $\frak L$
lie in $\sim L^2 S^{-1} k^{-3}$ planes.  Since the planes can be taken in general position, the lines $\frak L$
do not lie in an algebraic surface of lower degree.

(Our bound for the degree $d$ is the new ingredient in this section.  We will find the algebraic surface $Z$ by using
the polynomial cell decomposition of Theorem \ref{polycell}.  We initially tried to find $Z$ by
using the purely algebraic degree reduction argument from \cite{GK}, as in Section 3.  With this method, we proved that
a definite fraction of the lines of $Z$ lie in an algebraic surface of degree $L^2 S^{-1} k^{-2}$.
But this degree is too large to make our argument work.)

Now we begin the detailed proof of Proposition \ref{uniform}.

First we prove that almost all points of $\frak S$ lie in a surface $Z$ with controlled degree.  This lemma
is the most important step in the proof of Theorem \ref{incidence}.

\begin{lemma} \label{pointsinsurface} If the constant $A$ in inequality \ref{bigs} is sufficiently large, then
there is an algebraic surface $Z$ of degree $\lesssim L^2 S^{-1} k^{-3}$ that contains at least $(1 - 10^{-8})S$
points of $\frak S$.
\end{lemma}

\begin{proof} We let $\theta$ denote a large constant which we will choose later, and we let $d$
be the greatest integer less than $\theta L^2 S^{-1} k^{-3}$.  This
$d$ will be the degree of our surface $Z$.  First we check that $d \ge 1$.  By the Szemer\'edi-Trotter
theorem, $S \lesssim L^2 k^{-3} + L k^{-1}$.  But by inequality \ref{bigs}, $S \ge L k^{-1}$.  Therefore,
$S \lesssim L^2 k^{-3}$.  Hence we can choose $\theta$ so that $d \ge 1$.

Now we apply Theorem \ref{polycell} to construct a degree $d$ surface $Z$ such that ${\bf R^3} \setminus
Z$ is a union of $\sim d^3$ open cells $O_i$, each containing $\lesssim S d^{-3}$ points of $\frak S$.

Let us suppose that $Z$ contains $< (1 - 10^{-8}) S$ points of $\frak S$.  So the open cells $O_i$ all together contain
$\ge 10^{-8} S$ points of $\frak S$.  Since each cell contains $\lesssim S d^{-3}$ points of $\frak S$, there
must be $\gtrsim d^3$ cells that each contain $\gtrsim S d^{-3}$ points of $\frak S$.  We call these full cells.

We now prove an upper bound for $S$ using the cellular method from \cite{CEGSW}.

We let $\frak L(O_i)$ denote the subset of lines of $\frak L$ which intersect $O_i$.  We let $L_{cell}$ be the minimum of
$| \frak L (O_i)|$ among all the full cells $O_i$.  We apply the Szemer\'edi-Trotter inequality to the full cell with the fewest lines.
Since this full cell still contains $\gtrsim S d^{-3}$ points, we get the following inequality.

$$ S d^{-3} \lesssim L_{cell}^2 k^{-3} + L_{cell} k^{-1}. $$

Next we estimate $L_{cell}$ in terms of the degree of $Z$.  A line either lies in $Z$ or else it intersects $Z$ at most $d$ times.
Every time a line moves from one open cell $O_i$ to another, it needs to pass through $Z$.  
So each line of $\frak L$ intersects at most $d+1$ cells $O_i$.  So there are $\le L (d+1)$ pairs $(l, O_i)$ where $l \in \frak L(O_i)$.
But there are $\sim d^3$ full cells $O_i$.  Hence $L_{cell} \lesssim L d^{-2}$.  Plugging in this estimate for $L_{cell}$ we get
the following inequality.

$$ S d^{-3} \lesssim L^2 d^{-4} k^{-3} + L d^{-2} k^{-1} . $$

Recalling that $d \sim \theta L^2 S^{-1} k^{-3}$ and rearranging, we get the following inequality.

$$ S \le C( \theta^{-1} S + \theta L^3 S^{-1} k^{-4}). $$

Note that the constant $C$ does not depend on $\theta$.  (We could work it out explicitly using an explicit constant in Theorem
\ref{polycell} and in the Szemer\'edi-Trotter
theorem.)  At this point, we choose $\theta$ sufficiently large so that $C \theta^{-1} < 1/2$.  We can then move the term $C \theta^{-1} S$
to the left-hand side, and rearrange to get the inequality

$$ S \lesssim \theta^{1/2} L^{3/2} k^{-2}. $$

If the constant $A$ is sufficiently large, this inequality contradicts \ref{bigs}.  We conclude that there are less than $10^{-8} S$
points of $\frak S$ outside of $Z$.  

Finally, the degree of $Z$ is $d \le \theta L^2 S^{-1} k^{-3}$.  The constant $\theta$ is a particular number that we chose
above.  In particular $\theta$ does not depend on $A$.   And so $d \lesssim L^2 S^{-1} k^{-3}$ as desired.  \end{proof}

We let $\frak S_Z$ denote the points of $\frak S$ that lie in $Z$.  By Lemma \ref{pointsinsurface}, $ |\frak S \setminus
\frak S_Z| \le 10^{-8} S$.  Our next goal is to prove that many lines of $\frak L$ lie in the surface $Z$.  This result depends on a quick calculation about
the degree $d$.  Recall that an average line of $\frak L$ contains $S k L^{-1}$ points of $\frak S$.  We prove that
the degree $d$ is much smaller than $S k L^{-1}$.

\begin{lemma} \label{degest1} If the constant $A$ is sufficiently large, then

$$ d < 10^{-8} S k L^{-1}. $$

\end{lemma}

\begin{proof} Inequality \ref{bigs} can be rewritten as 

$$ 1 \le A^{-1} S L^{-3/2} k^2. $$

Squaring this, we see that

$$ d \le d A^{-2} S^2 L^{-3} k^4 \lesssim A^{-2} S k L^{-1}. $$

Now choosing $A$ sufficiently large finishes the proof.  \end{proof}

As an immediate corollary, we get the following lemma.

\begin{lemma} \label{linesinsurface} If $l$ is a line of $\frak L$ that contains at least $10^{-8} S k L^{-1}$ 
points of $\frak S_Z$, then $l$ is contained in $Z$.
\end{lemma}

\begin{proof} The line $l$ contains at least $10^{-8} S k L^{-1}$ points of $Z$.  Since $d > 10^{-8} S k L^{-1}$, the line $l$ must lie in the surface $Z$.
\end{proof}

Let $\frak L_Z $ denote the set of lines in $\frak L$ that are contained in $Z$.

\begin{lemma} The set $\frak L_Z$ contains at least $(1/200) L$ lines.
\end{lemma}

\begin{proof}  We assumed
that there are $\ge (1/100) L$ lines of $\frak L$ which each contain $\ge (1/100) S k L^{-1}$
points of $\frak S$.  Let $\frak L_0 \subset \frak L$ be the set of these lines.  We claim that
most of these lines lie in $\frak L_Z$.  Suppose that a line $l$ lies in $\frak L_0 \setminus \frak L_Z$.
It must contain at least $(1/100) S k L^{-1}$ points of $\frak S$.  But by Lemma \ref{linesinsurface},
it contains $ < 10^{-8} S k L^{-1}$ points of $\frak S_Z$.  Therefore, it must contain at least
$(1/200) S k L^{-1}$ points of $\frak S \setminus \frak S_Z$.  This gives us the following inequality.

$$ (1/200) S k L^{-1} |\frak L_0 \setminus \frak L_Z| \le I(\frak S \setminus \frak S_Z, \frak L_0 \setminus \frak L_Z). $$

Here we write $I$ to abbreviate the number of incidences between a set of points and a set of lines.

On the other hand, each point of $\frak S$ lies in at most $2k$ lines of $\frak L$, giving us an upper bound
on incidences:

$$ I(\frak S \setminus \frak S_Z, \frak L_0 \setminus \frak L_Z) \le 2 k |\frak S \setminus \frak S_Z| \le 2 \cdot 10^{-8} S k. $$

Comparing these two inequalities, we see that $|\frak L_0 \setminus \frak L_Z| \le 4 \cdot 10^{-6} L$, which
implies that $|\frak L_Z| \ge (1/200) L$. \end{proof}

We have now carried out the first step of our outline: we found a surface $Z$ of degree $\lesssim 
L^2 S^{-1} k^{-3}$ which contains a definite fraction of the lines from $\frak L$.

We now turn to the second step of our outline.  We will prove that $Z$ contains some planes, and that these planes
contain many lines of $\frak L$.  This step is closely based on the techniques in \cite{GK} and \cite{EKS}.
The paper \cite{EKS} contains a clear introduction to the techniques.  In particular, Section 2 of \cite{EKS} proves all of
the fundamental lemmas from algebraic geometry that we need.  

Each point of $\frak S_Z$ lies in at least $k$ lines of $\frak L$.  But such a point does not necessarily lie in
any lines of $\frak L_Z$.  Therefore we make the following definition.

We define $\frak S_Z'$ to be the set of points in $\frak S_Z$ that lie in at least three lines
of $\frak L_Z$.  

This subset is important because each point of $\frak S_Z'$ is a special point of the surface
$Z$: either a critical point or a flat point.  Let's recall the definitions of critical points and flat points.

The surface $Z$ is the vanishing set of a polynomial $p$.  The polynomial $p$ can be factored into 
irreducible polynomials $p = p_1 p_2 ...$.  We assume that each irreducible factor of $p$ appears only once.
Now a point $x \in Z$ is called critical if the gradient $\nabla p$ vanishes at $x$.  If $x \in Z$ is not critical, we say
that $x$ is regular.  In a small neighborhood of a regular point $x \in Z$, $Z$ is a smooth submanifold. 
We say that a regular point $x \in Z$ is flat if the second fundamental form of $Z$ vanishes at $x$.

\begin{lemma} \label{critflat} Each point of $\frak S_Z'$ is either a critical point or a flat point of $Z$.
\end{lemma}

\begin{proof} Let $x \in \frak S_Z'$.  By definition, $x$ lies in three lines which all lie in $Z$.  If $x$ is a
critical point of $Z$, we are done.  If $x$
is a regular point of $Z$, then all three lines must lie in the tangent space of $Z$ at $x$.  In particular,
the three lines are coplanar.  Let $v_1, v_2, v_3$ be non-zero tangent vectors of the three lines at $x$.
The second fundamental form of $Z$ vanishes in each of these three directions.  Since the second fundamental
form is a symmetric bilinear form on the 2-dimensional tangent space, it must vanish.  Therefore,
$x$ is a flat point of $Z$. \end{proof}

(See also \cite{EKS}, Proposition 4 and Proposition 6 for a more detailed proof.)

Lemma \ref{critflat} shows that the points of $\frak S_Z'$ are important.  Next we show that almost every point of
$\frak S$ lies in $\frak S_Z'$.

\begin{lemma} \label{pointsinsurface'} The set $\frak S \setminus \frak S_Z'$ contains at most $10^{-7} S$ points.
\end{lemma}

\begin{proof} Lemma \ref{pointsinsurface} tells us that $| \frak S \setminus \frak S_Z| < 10^{-8} S$.

Suppose $x$ is a point in $\frak S_Z \setminus \frak S_Z'$.  The point $x$ lies in at least $k$ lines from $\frak L$,
but it lies in at most two lines from $\frak L_Z$.  So $x$ lies in $\ge k-2$ lines of $\frak L \setminus \frak L_Z$.

$$ (k-2) | \frak S_Z \setminus \frak S_Z' | \le I ( \frak S_Z \setminus \frak S_Z' , \frak L \setminus \frak L_Z). $$

On the other hand, we showed in Lemma \ref{linesinsurface} that each line of $\frak L \setminus \frak L_Z$ contains
$\le 10^{-8} S k L^{-1}$ points of $\frak S_Z$.  Therefore

$$ I  ( \frak S_Z \setminus \frak S_Z' , \frak L \setminus \frak L_Z) \le I (\frak S_Z , \frak L \setminus \frak L_Z)
\le (10^{-8} S k L^{-1} ) L. $$

Combining these inequalities, and recalling that $k \ge 3$, we see that

$$ |\frak S_Z \setminus \frak S_Z'| \le 10^{-8} { k \over {k-2}} S \le 3 \cdot 10^{-8} S. $$

\end{proof}

We let $\frak S_{crit} \subset \frak S_Z'$ denote the critical points in $\frak S_Z'$ and we let $\frak S_{flat} \subset
\frak S_Z'$ denote the flat points of $\frak S_Z'$.  We call a line $l \subset Z$ a critical line of $Z$ if every point
of $l$ is a critical point of Z.  We call a line $l \subset Z$ a flat line if it is not a critical line and every regular point in $l$
is flat.  

Our next goal is to show that $Z$ contains many flat lines, which is a step to showing that $Z$ contains a plane.  In order
to do this, we show that the flat points of $Z$ are defined by the vanishing of certain polynomials.

\begin{lemma} \label{flat} Let $x$ be a regular point of $Z$.  
Then $x$ is flat if and only if the following three polynomial vectors vanish at $x$:

$$ \nabla_{e_j \times \nabla p} \nabla p \times \nabla p, j = 1, 2, 3. $$

\end{lemma}

Here, $e_j$ are the coordinate vectors of ${\bf R^3}$, and $\times$ denotes the
cross product of vectors.  Each vector above has three components, so we have a total of nine polynomials.
Each polynomial has degree $\le 3 d$.  For more explanation, see Section 3 of \cite{GK} or Section 2 of
\cite{EKS}.  In \cite{EKS}, they use a more efficient set of polynomials: only three polynomials.

To find critical or flat lines, we use the following simple lemmas.

\begin{lemma} \label{critline} Suppose that a line $l$ contains more than $d$ critical points of $Z$.
Then $l$ is a critical line of $Z$.  \end{lemma}

\begin{proof} At each critical point of $Z$, the polynomial $p$ and all the components of $\nabla p$
vanish.  Since $p$ has degree $d$, we conclude that $p$ vanishes on every point of $l$.  Since
$\nabla p$ has degree $d-1$, we conclude that $\nabla p$ vanishes on every point of $l$.  Hence
$l$ is a critical line of $Z$.
\end{proof}

\begin{lemma} \label{flatline} Suppose that a line $l$ contains more than $3d$ flat points
of $Z$.  Then $l$ is a flat line of $Z$.
\end{lemma}

\begin{proof} Let $x_1, ..., x_{3d+1}$ be flat points of $Z$ contained in $l$.  By Lemma \ref{flat}, each polynomial 
$\nabla_{e_j \times \nabla p} \nabla p \times \nabla p$ vanishes at $x_i$.  Since the degree of
these polynomials is $\le 3d$, we conclude that each of these polynomials vanishes on $l$.  Similarly,
$p$ vanishes on $l$.  Therefore, the line $l$ lies in $Z$ and every regular point in $l$ is a flat point.  But
by definition, $x_i$ are regular points of $Z$.  Therefore, $l$ is not a critical line, and it must be a flat line.
\end{proof}

Using these lemmas, we will prove that a definite fraction of the lines of $\frak L$ are either critical or flat.

We define $\frak L_Z'$ to be the set of lines of $\frak L_Z$ that contain at least $(1/200) S k L^{-1}$ points
of $\frak S_Z'$.  

\begin{lemma} Each line in $\frak L_Z'$ is either critical or flat.
\end{lemma}

\begin{proof}
Since every point of $\frak S_Z'$ is either critical or flat, each line in $\frak L_Z'$ contains
either $(1/400) S k L^{-1}$ critical points or $(1/400) S k L^{-1}$ flat points.  But by Lemma \ref{degest1}, 
$d \le 10^{-8} S k L^{-1}$.  So by Lemmas \ref{critline} and \ref{flatline}, each line
of $\frak L_Z'$ is either critical or flat. \end{proof}

Now we show that $\frak L_Z'$ contains a definite fraction of the lines of $\frak L$.

\begin{lemma} \label{manylines} The number of lines in $\frak L_Z'$ is $\ge (1/200) L$.
\end{lemma}

\begin{proof} Recall that we assumed in the statement of Proposition \ref{uniform} that there
are at least $(1/100) L$ lines of $\frak L$ that each contain $\ge (1/100) S k L^{-1}$ points of
$\frak S$.  We denote these lines by $\frak L_0 \subset \frak L$.

Suppose a line $l$ lies in $\frak L_0 \setminus \frak L_Z'$.  Then $l$ contains at least
$(1/100) S k L^{-1}$ points of $\frak S$.  But it contains less than $(1/200) S k L^{-1}$ points
of $\frak S_Z'$.  Therefore, it contains at least $(1/200) S k L^{-1}$ points of $\frak S \setminus \frak S_Z'$.
So we get the following inequality.

$$ (1/200) S k L^{-1} | \frak L_0 \setminus \frak L_Z' | \le I(\frak S \setminus \frak S_Z', \frak L_0 \setminus \frak L_Z').$$

But since each point of $\frak S$ lies in at most $2 k$ lines of $\frak L$,

$$ I(\frak S \setminus \frak S_Z', \frak L_0 \setminus \frak L_Z') \le I(\frak S \setminus \frak S_Z', \frak L) \le
2k |\frak S \setminus \frak S_Z'|. $$

Lemma \ref{pointsinsurface'} says that $| \frak S \setminus \frak S_Z'| \le 10^{-7} S$.  Assembling all our inequalities,
we see that

$$(1/200) S k L^{-1} |\frak L_0 \setminus \frak L_Z'| \le 2 k \cdot 10^{-7} S. $$

Simplifying this expression, we see that 

$$ | \frak L_0 \setminus \frak L_Z'| \le 4 \cdot 10^{-5} L. $$

So almost all the lines of $\frak L_0$ lie in $\frak L_Z'$.  In particular, $\frak L_Z'$ contains $\ge (1/200) L$ lines.
\end{proof}

Next we bound the number of critical lines in $Z$.

\begin{lemma} A surface $Z$ of degree $d$ contains $\le d^2$ critical lines.
\end{lemma}

This lemma follows from Bezout's theorem applied to $p$ and $\nabla p$.  See 
Proposition 3 in \cite{EKS}.

If the constant $A$ from inequality $4.1$ is sufficiently large, then $d^2$ will be much less than $L$.
We record this calculation in the next lemma.

\begin{lemma} \label{degest2} If $A$ is sufficiently large, then $d \le 10^{-4} L^{1/2}$.
\end{lemma}

\begin{proof} The inequality 4.1 implies that $1 \le A^{-1} S L^{-3/2} k^2$.  Therefore

$$d \le d A^{-1} S L^{-3/2} k^2 \lesssim A^{-1} L^{1/2} k^{-1}. $$

Choosing $A$ sufficiently large finishes the proof. \end{proof}

In particular, we see that $Z$ contains at most $d^2 < 10^{-8} L$ critical lines.  Since $\frak L_Z'$ contains
at least $(1/200) L $ lines, we see that most of these lines must be flat.  In particular, $\frak L_Z'$ contains
at least $(1/300) L$ flat lines of $Z$.

We are trying to prove that $Z$ contains some planes.  Let $Z_{pl}$ denote the union of all planes
contained in $Z$.
We let $\tilde Z$ denote the rest of $Z$ so that $Z = Z_{pl} \cup \tilde Z$.  In terms of polynomials.
$Z$ is the vanishing set of $p$.  The polynomial $p$ factors into irreducibles: $p = p_1 p_2 ...$.
Some of these factors have degree 1, and some factors have degree more than 1.  Each factor of
degree 1 defines a plane, and $Z_{pl}$ is the union of these planes.  The product of the remaining
factors is a polynomial $\tilde p$, and $\tilde Z$ is the zero-set of $\tilde p$.  
A line which lies in both $Z_{pl}$ and $\tilde Z$ is actually a critical line of $Z$.  So a flat line of
$Z$ lies either in $Z_{pl}$ or in $\tilde Z$, but not both.  A flat line of $Z$ that lies in $\tilde Z$ is
a flat line of $\tilde Z$.  The number of flat lines in a surface of degree $\le d$ is bounded by
the following lemma from \cite{EKS}.

\begin{lemma} (\cite{EKS}, Proposition 8) If $Z$ is an algebraic surface of degree $\le d$ with no planar component, then $Z$ 
contains $\le 3d^2$ flat lines.
\end{lemma}

We have seen that $\frak L$ contains at least $(1/300) L$ flat lines of $Z$.  But $\tilde Z$ contains only
$3 \cdot 10^{-8} L$ flat lines.  The rest of the flat lines lie in $Z_{pl}$.  In particular, $\frak L$ contains
at least $(1/400) L$ lines in $Z_{pl}$.

Finally, we observe that the number of planes in $Z_{pl}$ is $\le d \lesssim L^2 S^{-1} k^{-3}$.  So one of these
planes must contain $\gtrsim S k^3 L^{-1}$ lines of $\frak L$.  In other words, $B \gtrsim S k^3 L^{-1}$.

At several points in the argument, we needed $A$ to be sufficiently large.  We now choose $A$ large enough for those
steps.  We conclude that either  $S \le A L^2 k^{-3/2} + L k^{-1}$ or else $S \lesssim L B k^{-3}$.   This finishes the proof of 
Proposition \ref{uniform}. \end{proof}

Proposition \ref{uniform} is the heart of the proof of Theorem \ref{incidence}.   We are going to reduce the general case 
to Proposition \ref{uniform}.  First we remove the assumption that many lines have roughly the average number of points.

\begin{proposition} \label{incidence-} Let $k \ge 3$.  Let $\frak L$ be a set of $L$ lines in ${\bf R^3}$ with
$\le B$ lines in any plane.  Let $\frak S$ be a set of $S$ points so that each point meets
between $k$ and $2k$ lines of $\frak L$.

Then $S \le C [ L^{3/2} k^{-2} + L B k^{-3} + L k^{-1}] $.
\end{proposition}

\begin{proof}
Let $\frak L_1$ be the subset of lines in $\frak L$ which contain $\ge (1/100) S k L^{-1}$ points of
$\frak S$.  If $| \frak L_1| \ge (1/100) L$, then we have all the hypotheses of Proposition \ref{uniform},
and we may conclude

$$S \le C_0  [ L^{3/2} k^{-2} + L B k^{-3} + L k^{-1}]. $$

We are going to prove that $S$ obeys this same estimate, with the same constant, regardless of the size of $\frak L_1$.
The proof will go by induction on the number of lines.

From now on we assume that $| \frak L_1| \le (1/100) L$.  The lines in $\frak L_1$ contribute most of the incidences.
In particular, we have the following inequality.

$$ I (\frak S, \frak L \setminus \frak L_1) \le (1/100) S k L^{-1} \cdot L = (1/100) S k. $$

We define $\frak S' \subset \frak S$ to be the set of points with $\ge (9/10) k$ incidences
with lines of $\frak L_1$.  

If $x$ is in $\frak S \setminus \frak S'$, then $x$ lies in at least $k$ lines of $\frak L$, but less than
$(9/10) k $ lines of $\frak L_1$.  So $x$ lies in at least $(1/10) k$ lines of $\frak L \setminus \frak L_1$.
Therefore,

$$ (1/10) k |\frak S \setminus \frak S'|  \le I (\frak S \setminus \frak S', \frak L \setminus \frak L_1)
\le I (\frak S, \frak L \setminus \frak L_1) \le (1/100) S k. $$

Rearranging, we see that $| \frak S \setminus \frak S'| \le (1/10) S$, and so $| \frak S'| \ge (9/10) S$.

A point of $\frak S'$ has at least $(9/10) k$ incidences with $\frak L_1$ and at most
$2 k$ incidences with $\frak L_1$.  This is a slightly larger range than we have considered
before.  In order to do induction, we need to reduce the range.  We observe $\frak S' =
\frak S'_+ \cup \frak S'_-$, where $\frak S'_+$ consists of points with $\ge k$ incidences
to $\frak L_1$ and $\frak S'_-$ consists of points with $\le k$ incidences with $\frak L_1$.
We define $\frak S_1$ to be the larger of $\frak S'_+$ and $\frak S'_-$.  It has $\ge
(9/20) S$ points in it.

If we picked $\frak S_1 = \frak S'_+$ then we define $k_1 = k$.  If we picked
$\frak S_1 = \frak S'_-$ then we define $k_1$ to be the smallest integer
$\ge (9/10)k$.  Each point in $\frak S_1$ has at least $k_1$ and at most
$2 k_1$ incidences with lines of $\frak L_1$.  Also, $k_1$ is an integer
$\ge (9/10) k \ge 27/10$, so $k_1 \ge 3$.

The set of lines $\frak L_1$ and the set of points $\frak S_1$ obey all the hypotheses
of Theorem \ref{incidence-} (using $k_1$ in place of $k$ and using the same $B$).  There
are fewer lines in $\frak L_1$ than in $\frak L$.  Doing induction on the number of lines,
we may assume that our result holds for these sets.  If we denote $| \frak L_1| = L_1$
and $| \frak S_1 | = S_1$, we get

$$ S_1 \le C_0 [ L_1^{3/2} k_1^{-2} + B L_1 k_1^{-3} + L_1 k_1^{-1} ] . $$

Now $S \le (20/9) S_1$.  Also, $L_1 \le (1/100) L$.  And $k_1 \ge (9/10) k$.

Therefore,

$$ S \le (20/9) S_1 \le [(20/9) (1/100) (10/9)^3] C_0 [L^{3/2} k^{-2} + L B k^{-3} + L k^{-1}].$$

The bracketed product of fractions is $< 1$, and so $S$ obeys the desired bound.

\end{proof}

Finally, we can prove Theorem \ref{incidence}.

\begin{proof}[Proof of Theorem \ref{incidence}] Let $k \ge 3$.  Suppose that $\frak L$ is a set of $L$ lines with $\le B$ in any plane.  Suppose
that $\frak S$ is a set of points, each intersecting at least $k$ lines of $\frak L$.

We subdivide the points $\frak S = \cup_{j=0}^\infty \frak S_j$, where
$\frak S_j$ consists of the points incident to at least $2^j k$ lines and at most
$2^{j+1} k $ lines.  We define $k_j$ to be $2^j k$.  Then Theorem \ref{incidence-}
applies to $(\frak L, \frak S_j, k_j, B)$, and we conclude that

$$ | \frak S_j | \le C_0 [ L^{3/2} k_j^{-2} + L B k_j^{-3} + L k_j^{-1}] $$

$$ \le 2^{-j} C_0 [ L^{3/2} k^{-2} + L B k^{-3} + L k^{-1} ]. $$

Now $S \le \sum_j | \frak S_j| \le 2 C_0 [L^{3/2} k^{-2} + L B k^{-3} + L k^{-1}] $.

\end{proof}

\section{Appendix: The example of a square grid}

In this section, we return to Erd{\H o}s's example of a square grid of points.  When $P$ is
a square grid of $N$ points, we show that $|Q(P)| \gtrsim N^3 \log N$ and $|G_k(P)| \gtrsim N^3 k^{-2}$
for all $2 \le k \le N/2000$.  So the estimates in Propositions \ref{quadbound} and \ref{partsymmbound} are
sharp up to constant factors.  We also study the set of lines $\frak L$ associated to a square grid $P$.
This set of lines shows that many of our incidence estimates are
sharp up to constant factors.

Let $S \ge 1$ be an integer.  Let $P$ be the grid of points $(x,y)$ where $x$ and $y$ are integers with
norm $\le 2S$.  Note that the number of points in $|P|$ is $N = (2 S + 1)^2$.  
Let $\frak L$ be the set of lines in ${\bf R^3}$ associated to the set $P$, as in
Section 2.

\begin{lemma} If $a,b,c,$ and $d$ are positive integers with norm $\le S$, then the line from
$(a,b,0)$ to $(c,d,1)$ is contained in $\frak L$.
\end{lemma}

\begin{proof} Using the parametrization in \ref{Param}, we see that the line from $(a,b,0)$ to
$(c,d,1)$ is the line $L_{pq}$, where $p$ and $q$ are defined by the following equations.

$$ (1/2) (p_x + q_x) = a; (1/2) (p_y + q_y) = b; (1/2) (q_y - p_y) = c - a ; (1/2) (p_x - q_x) = d - b. $$

Solving these equations, we get $p = (a + d - b, b - c + a)$ and $q = (a - d + b, b + c - a)$.  Since 
$a,b,c$, and $d$ are positive integers of norm $\le S$, it follows that $p_x, p_y, q_x,$ and $q_y$ are
integers of norm $\le 2S$, and so $p$ and $q$ lie in $P$. \end{proof}

Let $\frak L_0 \subset \frak L$ be the set of lines from $(a,b,0)$ to $(c,d,1)$ where $a,b,c$, and $d$
are positive integers with norm $\le S$.  In the proposition below, we study the incidences of
$\frak L_0$.  Note that $| \frak L_0 | = S^4$.

\begin{proposition} \label{linesandgrids} Let $\frak S_k$ be the set of points in ${\bf R^3}$ that lie
in at least $k$ lines of $\frak L_0$.  For any $k$
in the range $2 \le k \le (1/400) S^2 $, $|\frak S_k| \gtrsim S^6 k^{-2}$.
\end{proposition}

\begin{proof}
Consider a point $x$ in ${\bf R^3}$ contained in the slab $0 < x_3 < 1$.  We define a map
$F_x: {\bf R^2} \rightarrow {\bf R^2}$ by saying that $F_x(a,b) = (c,d)$ if the line
from $(a,b,0)$ through $x$ hits $(c,d,1)$.  We define $G$ to be the integral grid in the plane
given by $\{ (a,b) \}$ with $1 \le a,b \le S$.  The number of lines from $\frak L_0$ which pass
through $x$ is exactly the cardinality of $F_x(G) \cap G$.  Now any intersection of two lines
from $\frak L_0$ will have rational coordinates, so we can assume the coordinates of $x$ are rational.
Let us say that the $x_3$ coordinate of $x$ is $p/q$, written in lowest terms.  

By a similar triangles argument, $F_x (G)$ is a square grid with spacing $\frac{q-p}{p}$.  Since $p$ and
$q$ are in lowest terms, the intersection $F_x (G) \cap G$ will be a rectangular grid with spacing $q-p$.
The edges of this rectangle will have length $< S$.  So the number
of points in $F_x(G) \cap G$ is at most $S^2 (q-p)^{-2}$.  On the other hand, the edges of this rectangle
have length  $< S {{q-p} \over {p}}$.  Therefore, the number of points of $F_x(G) \cap G$ is at most $S^2 p^{-2}$.
Combining these estimates, we see that $|F_x(G) \cap G| \le 4 S^2 q^{-2}$.

Let us say that the middle half of $G$, written $G_{middle} \subset G$, is the integral grid $\{ (a,b) \}$ with $(1/4) S \le a,b
\le (3/4) S$.  If $F_x$ maps a vertex from $G_{middle}$ into $G$, then the number
of intersections between $F_x(G)$ and $G$ is fairly close to this upper bound.  Using the arguments from the
last paragraph, it's straightforward to show that $|F_x(G) \cap G| \ge (1/100) S^2 q^{-2}$ whenever $F_x(G_{middle})$
intersects $G$.

Let us define $X(p,q)$ to be the set of $x = (x_1, x_2, p/q)$ so that $F_x(G_{middle}) \cap G$ is non-empty.
The set $X(p,q)$ lies in $\frak S_k$ whenever $k \le (1/100) S^2 q^{-2}$.  Equivalently, $X(p,q)$ lies in
$\frak S_k$ whenever $q \le (1/10) S k^{-1/2}$.

For any pair of points $(a_1, b_1) \in G_{middle}$ and $(a_2, b_2) \in G$, there is a unique $x \in X(p,q)$ so
that $F_x(a_1, b_1) = (a_2, b_2)$.  There are $\sim S^4$ such pairs of points.  
Each element of $X(p,q)$ corresponds to at least $(1/100) S^2 q^{-2}$ pairs
of points, and at most $4 S^2 q^{-2}$ pairs of points.  Therefore $|X(p,q)| \sim S^2 q^2$.

Now we fix $k \le (1/400) S^2$.  We pick $q$ in the range $(1/20) S k^{-1/2} \le q \le (1/10) S k^{-1/2}$.
Because $k$ is not too big, this range of $q$ contains some integers.  
For each $p$ coprime to $q$, $X(p,q) \subset \frak S_k$.  The sets $X(p,q)$ are clearly disjoint, and so

$$|\frak S_k| \gtrsim \sum_{q = (1/20) S k^{-1/2}}^{(1/10) S k^{-1/2}} \sum_{0 < p < q, gcd(p,q) = 1} |X(p,q)|
\gtrsim \sum_{q = (1/20) S k^{-1/2}}^{(1/10) S k^{-1/2}} \phi(q) S^2 q^2. $$

The sums of the Euler totient function $\phi(n)$ are well studied.  Theorem 3.7 in \cite{A} gives the
asymptotic $\sum_{q=1}^x \phi(q) = {3 \over {\pi^2}} x^2 + O(x \log x)$.  Therefore, $\sum_{q=x}^{2x}
\phi(q) \sim x^2$.  Therefore,

$$ | \frak S_k| \gtrsim \left( S k^{-1/2} \right)^2 S^2 q^2 \sim S^6 k^{-2} . $$

 \end{proof}

Recall that $|G_k(P)|$ is at least $|G_k'(P)|$, which is the number of points lying in
at least $k$ lines of $\frak L$.  So we see that $|G_k(P)| \gtrsim S^6 k^{-2} \sim N^3 k^{-2}$ for
all $2 \le k \le (1/400) S^2 \le N/2000$.  The equation \ref{quadformula} gives

$$|Q(P) | \sim \sum_{k=2}^N k |G_k(P)| \gtrsim \sum_{k=2}^{N/2000} N^3 k^{-1} \gtrsim N^3 \log N. $$

Now we consider how sharp our incidence theorems are.
The set of lines $\frak L_0 \subset \frak L$ has $\lesssim N \sim S^2$ lines in any plane or doubly ruled surface
by Proposition \ref{genericity}.
This example shows that Theorems \ref{theorem1} and \ref{theorem2} are sharp up to constant factors.

Next we consider Theorem \ref{incidence}.  The lines $\frak L_0$ correspond to Example 3 in Section 4.  The three
examples show that Theorem \ref{incidence} is sharp up to constant factors as long as $B \gtrsim L^{1/2}$.  The example
$\frak L_0$ has $B \sim L^{1/2}$.  For much
smaller values of $B$, we don't know what happens.  For example, suppose that $\frak L$ is a set of $L$ lines in ${\bf R^3}$
with at most 100 lines in any plane.  How many points can be incident to three lines of $\frak L$?  Or suppose that
$\frak L$ is a set of $L$ lines in ${\bf R^3}$ with at most 100 lines in any plane or doubly ruled surface.  How many points
can be incident to two lines of $\frak L$?

\vfill
\eject

\tiny

\textsc{L. GUTH, SCHOOL OF MATHEMATICS, INSTITUTE FOR ADVANCED STUDY, PRINCETON NJ}

{\it lguth@math.ias.edu}

\bigskip

\textsc{N. KATZ, DEPARTMENT OF MATHEMATICS, INDIANA UNIVERSITY, BLOOMINGTON IN}

{\it nhkatz@indiana.edu}

\end{document}